\documentclass{article}
\usepackage{amsmath}
  \usepackage{paralist}
  \usepackage{graphics} 
  \usepackage{epsfig} 
 \usepackage[colorlinks=true]{hyperref}

\usepackage{color}

\definecolor{myred}{rgb}{1,0.5,0}

\hypersetup{urlcolor=blue, citecolor=red}

\newtheorem{thm}{Theorem}[section]
\newtheorem{proof}{Proof}[section]
\newtheorem{lemma}[thm]{Lemma}
\newtheorem{rem}[thm]{Remark}
\newtheorem{corollary}[thm]{Corollary}

\definecolor{myred}{rgb}{1,0.5,0}

\title
      {Neumann boundary feedback stabilization for a
      nonlinear wave equation: A strict $H^2$-Lyapunov function}

\author{
Martin Gugat$^a$,
G\"unter Leugering$^a$
and
Ke Wang$^{a,b}$}

\date{ }

\begin{document}

\maketitle

{\footnotesize \centerline{$^a$Lehrstuhl 2 f\"ur Angewandte
Mathematik} \centerline{Friedrich-Alexander-Universit\"at
Erlangen-N\"urnberg} \centerline{Cauerstr. 11, 91058 Erlangen,
Germany}}

\medskip

{\footnotesize
 \centerline{$^b$School of Mathematical Sciences, Fudan University, Shanghai 200433,
 China}}



\thanks{This work was
supported by DFG
in the framework of the Collaborative Research Centre
CRC/Transregio 154,
Mathematical Modelling, Simulation and Optimization Using the Example of Gas Networks,
project C03 and A05.}


\begin{abstract}
For a system that is governed by
the isothermal Euler equations with friction
for ideal gas,
the corresponding  field of characteristic curves
is determined by the velocity of the flow.
This velocity
is determined by
a second-order quasilinear hyperbolic equation.
For the
corresponding
initial-boundary value problem with
Neumann-boundary feedback, we consider non-stationary solutions
locally around a stationary state on a finite time interval and
discuss the well-posedness of this kind of problem. We introduce a
strict $H^2$-Lyapunov function and show that the boundary feedback
constant can be chosen such that the $H^2$-Lyapunov function and
hence also the $H^2$-norm of the difference between the
non-stationary and the stationary state decays exponentially with
time.
\end{abstract}

\section{Introduction}
The flow of gas through a pipeline
is modelled by the isothermal Euler
equations with friction.
In the operation of gas pipelines,
it is essential that the velocities
remain below critical values where
vibrations occur and noise is created, see
\cite{zou}.
We study a  quasilinear wave equation
for the gas velocity
in the case of ideal gas which is derived from the isothermal Euler
equations with friction.
Using Neumann
feedback at one end of the pipe,
 we stabilize the  solution of the corresponding
initial-boundary value problem
with homogeneous Dirichlet boundary conditions at the other end of the pipe
to a desired subsonic stationary state.
Except for its nonlinearity, this system is of a similar form  as the system with the linear
wave equation which has been studied for example in  \cite{ko:rapid}.

The first results
on the boundary feedback stabilization for a quasilinear wave equation
 have been obtained by M. Slemrod in
\cite{Slemrod} and J. Greenberg \& T. Li in \cite{GL}
by using the method of characteristics.
In \cite{CBD1}, J.-M. Coron, B. d'Andrea-Novel \& G. Bastin constructed a strict
$H^2$-Lyapunov function for the boundary control of hyperbolic systems
of conservation laws without source term.
In \cite{CBD2}, they constructed a strict $H^2$-Lyapunov
for quasilinear hyperbolic systems with dissipative boundary conditions
 without source term.
 More recently in \cite{CoronBastin2015}, Coron and Bastin
 study the Lyapunov stability of the $C^1$-norm for quasilinear
 hyperbolic systems of the first order.
 They consider $W^1_p$--Lyapunov functions for $p<\infty$ and
 look at the limit for $p\rightarrow \infty$.



Based upon \cite{CBD1}, M. Dick, M. Gugat \& G. Leugering considered
the isothermal Euler equations with friction
with Dirichlet boundary  feedback  at both ends of the system
 and introduced a strict
$H^1$-Lyapunov function, which is a weighted and squared $H^1$-norm
of the difference between the nonstationary and the stationary
state. They developed Dirichlet boundary feedback conditions which guarantee
that the $H^1$-norm of the difference between the non-stationary and
the stationary state decays exponentially with time (see
\cite{DGL1}).
In \cite{GLTW}, we have defined a  strict
$H^2$-Lyapunov function for this stabilization problem.
In contrast to
\cite{CBD1}, \cite{DGL1} and \cite{GLTW}
 in the present paper
a Neumann boundary feedback law is used at one end of the interval
for the stabilization of the system.
This is motivated by the nice properties
of the corresponding Neumann feedback for the linear wave equation
that leads to finite-time stabilization for a certain
feedback parameter, see  \cite{ko:rapid}, \cite{fatiha}.


In our paper, by constructing a strict
$H^2$-Lyapunov function and choosing suitable boundary feedback
conditions,
we give results about the boundary feedback
stabilization for a second-order quasilinear hyperbolic
equation with source term.
The exponential decay of the solution of a second-order
quasilinear hyperbolic equation is established. This solution
measures the difference between the present state and a desired
stationary state, which is in general not constant for our system.


This paper is organized as follows: In Section \ref{isothermal} we consider the
isothermal Euler equations both in  physical variables and
in terms of Riemann invariants. Then we transform the
isothermal Euler equations to a second-order quasilinear hyperbolic
equation. In Section \ref{well-posed}
we state a result about the well-posedness of
general second-order quasilinear hyperbolic systems on a finite time
interval (see Lemma \ref{well-posedness}).
Our main results about the exponential
decay of the $H^2$-norm and $C^1$-norm are presented in Theorem \ref{thm4.2}
and Corollary \ref{cor4.3} in Section \ref{sec4}. The proofs of Theorem \ref{thm4.2} and
Corollary \ref{cor4.3} are given in Section \ref{sec5}.
The infinite time horizon case is studied in Section \ref{sec6}.
We show that due to the stabilization, the solution exists globally in time.

\section{The isothermal Euler equations and  a quasilinear wave equation}

\label{isothermal}

In this section, we present the isothermal Euler equations with
friction for a single pipe both in terms of the physical variables
and in terms of Riemann invariants.

Let a finite time $T>0$ be given.
The system dynamics for the gas flow in a single pipe can be modeled
by a hyperbolic system, which is described by the isothermal Euler
equations (see \cite{BHK1},\cite{BHK2},\cite{DGL1}):
\begin{align}\label{2.1}
&\rho_t+q_x=0,\\\label{2.2}
&q_t+\Big(\frac{q^2}\rho+a^2\rho\Big)_x=-\frac{f_g}{2\delta}\frac{q|q|}{\rho},
\end{align}
where $\rho=\rho(t,x)>0$ is the density of the gas, $q=q(t,x)$ is
the mass flux, the constant $f_g > 0$ is a friction factor,
$\delta>0$ is the diameter of the pipe and $a > 0$ is the sonic
velocity in the gas. We consider the equations on the domain
$\Omega:=[0,T]\times[0,L]$.
Equation (\ref{2.1}) states the conservation of mass and equation
(\ref{2.2}) is the momentum equation. We use the notation
$$\theta=\frac{f_g}{\delta}.$$
In this paper, we consider positive gas flow in subsonic or
subcritical states, that is,
\begin{align}\label{2.3}
0<\frac{q}\rho < a.
\end{align}

The isothermal Euler equations
(\ref{2.1}) and (\ref{2.2}) give rise to the second-order
equation
\begin{align}\label{2.21}
\tilde u_{tt} + 2\tilde u \, \tilde u_{tx}-(a^2-\tilde u^2)
\, \tilde
u_{xx}=\tilde F(\tilde u,\tilde u_x,\tilde u_t),
\end{align}
where $\tilde u$ is the unknown function
and satisfies
\begin{equation}
\label{2.35}
\tilde u=\frac{q}\rho,
\end{equation}
that is $\tilde u$ is the velocity of the gas.
The lower order term
is
\begin{align}\label{2.22}
\tilde F(\tilde u,\tilde u_x,\tilde u_t)=-2\tilde u_t\tilde
u_x-2\tilde u\tilde u_x^2
-\frac{3}{2}\theta \tilde u
|\tilde u|
\tilde u_x
-\theta |\tilde u| \tilde u_t.
\end{align}
From the velocity $\tilde u$,
the
density  $\rho$ can be obtained from
 the initial value and the differential equation
\begin{align}\label{2.34}
(\ln\rho)_t=\frac{1}{a^2}\, \left(\tilde u\tilde
u_t
+(\tilde u^2-a^2)\, \tilde u_x
+\frac{1}2\theta \, |\tilde u| \tilde u^2\right).
\end{align}
Then $q$ can be obtained from the equation
$q = \rho \, \tilde u$.

To stabilize the system governed by
the quasilinear wave equation (\ref{2.21}) locally around a given
stationary state $\bar u(x)$, we use the boundary feedback law
\begin{align*}
&
\tilde u_x(0) =\bar u_x(0)+k \, \tilde u_t(0),\\
&
\tilde u(L) =\bar u(L),
\end{align*}
with a feedback parameter $k\in (0,\infty)$.

In terms of the physical variables $(q,\rho)$,
the boundary feedback law is
\begin{align*}
{\rm at} \,&x=0: q_x  - ({\rm ln}(\rho))_x\, q = \rho \, \bar u_x(0)+k
\left[ q_t  -  ({\rm ln}(\rho))_t\, q \right],\\
{\rm at} \,&x=L:q  =\bar u(L) \; \rho.
\end{align*}

%

Sufficient conditions for the exponential stability of this system will
be presented in Theorem \ref{thm4.2} in Section \ref{sec4}.


\subsection{
The Riemann invariants and
a differential equation for $\rho$ in terms of the velocity}
For classical solutions the isothermal Euler equations (\ref{2.1})
and (\ref{2.2}) can be equivalently written as the following system
\begin{align}\label{2.4}
\partial_t\left(\begin{array}{ll}
\rho\\q\end{array}\right)+\hat{A}(\rho,q)\
\partial_x\left(\begin{array}{ll}
\rho\\q\end{array}\right)=\hat{G}(\rho,q)
\end{align}
with the matrix
\begin{align*}
\hat{A}(\rho,q):=\left(\begin{array}{cc}
0 & 1\\
a^2-\frac{q^2}{\rho^2} & 2\frac{q}{\rho}\end{array}\right)
\end{align*}
and the source term
\begin{align*}
\hat{G}(\rho,q):=\left(\begin{array}{cc}
0\\-\frac{\theta}2\frac{|q|\, q}{\rho}\end{array}\right).
\end{align*}
System (\ref{2.4}) has two eigenvalues
$\tilde{\lambda}_-(\rho,q)$,
$\tilde{\lambda}_+(\rho,q)$
 and in the subsonic case we have
\begin{align}\label{2.5}
\tilde{\lambda}_-(\rho,q)=\frac{q}{\rho}-a<0<\tilde{\lambda}_+(\rho,q)=\frac{q}{\rho}+a.
\end{align}
In terms of the  Riemann invariants $R_\pm=R_{\pm}(\rho, q) =-\frac{q}\rho\mp
a\ln(\rho)$ the system (\ref{2.4})
has the
diagonal form
\begin{align}\label{2.6}
\partial_t\left(\begin{array}{ll}
R_+\\R_-\end{array}\right)+\hat{D}(R_+,R_-)\
\partial_x\left(\begin{array}{ll}
R_+\\R_-\end{array}\right)=\hat{S}(R_+,R_-),
\end{align}
where
\begin{align*}
\hat D(R_+,R_-)&:=\left(\begin{array}{cc}
\tilde{\lambda}_+ & 0\\
0 & \tilde{\lambda}_- \end{array}\right)=\left(\begin{array}{cc}
-\frac{R_++R_-}{2}+a & 0\\
0 & -\frac{R_++R_-}{2}-a\end{array}\right),
\end{align*}
\begin{equation}
\label{shutdefinition}
\hat S(R_+,R_-):=
\frac\theta8  \, (R_+ + R_-) \, |R_+ + R_-| \,
\left(\begin{array}{ll}1\\1\end{array}\right).
\end{equation}
In terms of $R_\pm$, for the physical variables
 $\rho$ and $q$
 we have
\begin{align}\label{2.9}
&\rho=\exp\left(\frac{R_--R_+}{2a}\right),\\\label{2.10}
&q=-\frac{R_++R_-}{2}\exp\left(\frac{R_--R_+}{2a}\right).
\end{align}
A gas flow is positive and subsonic (i.e. $0<q/\rho < a$) if and
only if
\begin{align}\label{2.11}
-2a < R_+(t, x) + R_-(t, x) < 0
\; \mbox{\rm for all }\; (t, x)\in\Omega.
\end{align}

For the velocity  $\tilde u=\tilde u(\rho,q)$  defined
in (\ref{2.35}) we have
\begin{align}\label{2.18}
\tilde u
=\frac{R_++R_-}{-2}.
\end{align}
Due to (\ref{2.5}),
we can express the velocity in terms of the eigenvalues as
\begin{align*}
\tilde u=\frac{\tilde{\lambda}_++\tilde{\lambda}_-}2.
\end{align*}
Due to equation  (\ref{2.18}),
(\ref{2.6}) yields
the second-order equation
(\ref{2.21}).
A detailed derivation can be found in \cite{gubook}.
The second-order quasilinear equation
(\ref{2.21}) is hyperbolic with the eigenvalues
\begin{align}\label{2.23}
\tilde\lambda_-=\tilde u-a<0<\tilde\lambda_+=\tilde u+a.
\end{align}
Using the isothermal Euler equations (\ref{2.1}) and (\ref{2.2}), we
obtain the partial derivatives of $\tilde u$ with respect to $t$ and
$x$, respectively,
\begin{align*}
\tilde u_t&
=\frac{q_t}{\rho}-\frac{q\rho_t}{\rho^2}\\
&=-\frac{1}{\rho}\left(\frac{q^2}{\rho}+a^2\rho\right)_x-\frac{q\rho_t}{\rho^2}
-\frac{\theta}2\frac{q \,|q|}{\rho^2}\\
&=\tilde u\frac{\rho_t}{\rho}+(\tilde
u^2-a^2)\frac{\rho_x}{\rho}-\frac{\theta}2\tilde u \,|\tilde u|
\end{align*}
and
\begin{align*}
\tilde
u_x&=\frac{q_x}{\rho}-\frac{q\rho_x}{\rho^2}=-\frac{\rho_t}{\rho}-\tilde
u\frac{\rho_x}{\rho}.
\end{align*}
Multiplying $\tilde u_t$ and $\tilde u_x$ by $\tilde u$ and $\tilde
u^2-a^2$, respectively,
by adding the two equations we obtain (\ref{2.34}), which means
that  $\rho$  and $ q$ can be obtained from $\tilde u$ and the initial
data.
Note that since $\tilde u = \frac{q}{\rho}$,
we have the same value for $\tilde u$ for
$\lambda q$ and $\lambda \rho$ where $\lambda\in (0,1]$.
So we cannot expect to recover the values of
$(q,\rho)$ from $\tilde u$  without additional information on $(q,\rho)$.
%
%
In a similar way as  (\ref{2.34}), we obtain the equation
\begin{equation}
\label{1672015}
{\rm ln}(\rho)_x = - \frac{1}{a^2} \left(\tilde u_t + \tilde u \, \tilde u_x +
\frac{\theta}{2} |\tilde u| \, \tilde u \right).
\end{equation}
Thus if $\tilde u$ is known, the values of
$\rho$ can be determined from the value of $\rho$
at a boundary point ($x=0$ or $x=L$)
and (\ref{1672015}) by integration.


\subsection{Stationary states of the system}

In \cite{DGL2} the existence, uniqueness and
the properties of stationary subsonic $C^1$-solutions
$(\bar\rho(x),\bar q(x))$ of the isothermal Euler equations
have been discussed.
The stationary states of the system on networks are studied in \cite{nhm}.

%
Here we focus on the stationary states of (\ref{2.21}).
Let  $\bar u=\bar{u}(x)$
denote a stationary
state for the second-order equation (\ref{2.21}).
Then
(\ref{2.21}) yields the following  second-order ordinary
differential equations for $\bar u(x)$:
\begin{align}\label{2.24}
(a^2-\bar u^2(x))\frac{d^2}{dx^2}\bar u(x)=2\bar
u(x)\Big(\frac{d}{dx}\bar u(x)\Big)^2+\frac32\theta\bar
u(x)\, |\bar u(x)| \,  \frac{d}{dx}\bar u(x).
\end{align}
This implies that
equation
(\ref{2.21}) has
constant stationary states
$\bar u \in (-\infty,\infty)$  that can attain arbitrary real values.
In contrast to this situation,
 the isothermal Euler equations with friction
(that is (\ref{2.1}), (\ref{2.2}))
 do not have constant stationary states
 except for the case of velocity zero.
 The stationary states of  (\ref{2.1}), (\ref{2.2})
 have been  studied
 in \cite{nhm}.
Now we consider the question:
Given a constant state $\bar u=\lambda\in (0,\,\infty)$,
is there a solution $(q,\rho)$ of
(\ref{2.1}), (\ref{2.2}) that corresponds to the constant velocity $\bar u$?
For $\lambda=0$ we obtain the constant solution of
(\ref{2.1}), (\ref{2.2}) where $q=0$.
For $\lambda>0$ there is a corresponding solution  of travelling wave type
(in particular the corresponding solution of (\ref{2.1}), (\ref{2.2}) is not
stationary),
namely
\begin{equation}
\left(q(t,x),\rho(t,x)\right)= \left( \lambda \, \alpha( \lambda \, t- x),\;
\alpha( \lambda t- x)\right)
\end{equation}
where the function $\alpha$ is given by
\begin{equation}
\label{alphadefinition}
\alpha(z)= C \, \exp\left(\frac{\lambda^2 \theta}{2 \, a^2 } z \right)
\end{equation}
and $C>0$ is a positive constant.
%
Equation (\ref{2.24}) can be rewritten in
the form
\begin{equation}
\frac{d}{dx}\left( (a^2-\bar u^2(x)) \bar u_x(x) - \frac{\theta}{2}
|\bar u(x)| \,
\bar u^2(x)\right)=0.
\end{equation}


Thus for every stationary state $\bar u$ of (\ref{2.21}) there exists
a constant $\lambda \in (-\infty,\,\infty)$ such that
$\bar u$ satisfies the first order ordinary
differential equation
\begin{equation}
\label{lambdagleichung}
 (a^2-\bar u^2(x))
 \bar u_x(x)  = \lambda+  \frac{\theta}2  | \bar u| \, \bar u^2(x).
\end{equation}



We use the notation $\bar u_0:=\bar u(0)$. Assume that
$\bar u_0 \in (0,\,a)$.
Let $[0,\,x_0)$ denote the maximal existence interval of the solution.
For the  solutions that are not constant, we have
two cases:

If $\lambda + \frac{\theta}{2}\;  \bar u_0^3>0$,
$\bar u$ is strictly increasing on $[0,\,x_0)$ and
\begin{align*}
\lim_{x\rightarrow x_0-}\bar u(x)=a,
\lim_{x\rightarrow x_0-}\frac{d}{dx}\bar u(x)=+\infty,\
\lim_{x\rightarrow x_0-}\frac{d^2}{dx^2}\bar u(x)=+\infty\,
.
\end{align*}

If $\lambda + \frac{\theta}{2}\;  \bar u_0^3<0$,
$\bar u$ is strictly decreasing  on $[0,\,x_0)$ and
\begin{align*}
\lim_{x\rightarrow x_0-}\bar u(x)= - a,
\lim_{x\rightarrow x_0-}\frac{d}{dx}\bar u(x)=-\infty,\
\lim_{x\rightarrow x_0-}\frac{d^2}{dx^2}\bar u(x)=-\infty\,
.
\end{align*}

For stationary $\rho$ and $\tilde u$,
equation (\ref{2.34}) implies
(\ref{lambdagleichung})
 with $\lambda=0$.
The stationary states that correspond to $\lambda\not=0$
cannot be deduced from the stationary states of (\ref{2.1}), (\ref{2.2}).
Thus all the stationary solutions of (\ref{2.21})
that correspond to a  stationary state of (\ref{2.1}), (\ref{2.2})
must satisfy the equation
\begin{equation}
\label{statvor}
\bar u'(0)= \frac{\theta}{2} \frac{|\bar u_0| \,\bar u_0^2}{a^2 - \bar u_0^2}.
\end{equation}
The following Lemma contains an explicit representation
for these stationary velocities.
\begin{lemma}
\label{lemmastationary}
Let a  subsonic stationary state $\bar u(x)>0$ for $x\in [0,L]$ that
is not constant and satisfies
(\ref{statvor})
  be given.
Let $W_{-1}(x)$ denote the real branch of the Lambert W--function
(see \cite{CGH, LA})
with $W_{-1}(x) \leq -1$.
Then the following equation holds for all $x\in [0,L]$:
\[
 (\bar u(x))^ 2=
\frac{a^2}{- W_{-1} ( - \exp (\theta \, x + \bar C))},
\]
where $\bar C$ is a real constant such that
$ \bar C \leq - 1-  \theta L $.
\end{lemma}

\begin{proof}
Separation of variables yields
\[
x  + \hat C= \int \frac{a^2 - \bar u(x)^2}{\frac{\theta}{2} \bar u(x)^3} \,
\bar u'(x)\,
 dx
=
-\frac{1}{\theta }
\left[{\rm ln}(a^2) +
{\rm ln}\left( \frac{\bar u(x)^2}{a^2}\right)  + \frac{a^2}{\bar u(x)^2}
\right].
\]
Define $\xi =  \frac{a^2}{\bar u(x)^2} \in (1,\,\infty)$.
We have $\xi + {\rm ln}(1/\xi)=\xi - {\rm ln}(\xi).
$
Thus
\[
- \exp(\theta x + \theta  \hat C + {\rm ln}(a^2)) = - \exp (-\xi +{\rm ln}(\xi))
=(-\xi) \exp( - \xi).
\]
Now the definition of $W_{-1}$ as the inverse function
of $z\exp(z)$ for $z \in (-\infty,\, -1)$
yields the assertion.
\end{proof}
Since for the stationary states $(q,\rho)$ of (\ref{2.1}), (\ref{2.2})
the flow rate $q$ is constant,  by (\ref{2.35}) we get
the corresponding density as $\rho(x)= \frac{q}{\bar u(x)}$.

%
%


\section{Well-posedness  of the system locally around stationary states}

\label{well-posed}

Now we consider  non-stationary solutions  locally around a subsonic stationary state $\bar
u(x) >0$ on $\Omega$
that satisfies
(\ref{lambdagleichung}) with $\lambda=0$, that is
that corresponds to a stationary state of  (\ref{2.1}), (\ref{2.2}).
For a solution $\tilde u(t,x)$ of (\ref{2.21}),
define
\begin{equation}
\label{udefinition2016}
u (t,x) = \tilde u(t,x) - \bar u(x).
\end{equation}
Then (\ref{2.21}),
(\ref{2.24})
and (\ref{lambdagleichung})
 yield the equation
\begin{align}\label{2.26}
u_{tt}+2(\bar u+u)u_{tx}-\Big(a^2-(\bar u+u)^2\Big)u_{xx}=
F(x,u,u_x,u_t),
\end{align}
where $F:=F(x,u,u_x,u_t)$ satisfies
\begin{eqnarray}
\label{Fdefinition}
F & = &
 \tilde F(u + \bar u,\, u_x + \bar u_x\,,u_t)
+
\frac{ a^2 - (\bar u + u)^2}{a^2 - \bar u^2} \, \bar u \left(  2 (\bar u_x)^2 + \frac{3}{2}\, \theta \, |\bar u|\, \bar u_x  \right)
\\
& = &
 \tilde F(u + \bar u,\, u_x + \bar u_x\,,u_t)
 -
 \frac{ a^2 - (\bar u + u)^2}{a^2 - \bar u^2} \, \tilde F( \bar u,\,  \bar u_x\,,0).
\end{eqnarray}
If $\bar u \geq 0 $ and $\bar u+ u\geq 0$,
we have
\[\tilde F(u + \bar u,\, u_x + \bar u_x\,,u_t)
=
\tilde F(\bar u, \bar u_x, 0)
+
\tilde F(u,u_x,u_t)
\]
\[
-2 \, \bar u  \, u_x^2
- 2 \,  \bar u_x \, u_t
- 4  \, \bar u_x\, u \, u_x  - 2 \, \bar u_x^2\, u
- 4 \, \bar u\, \bar u_x \, u_x
\]
\[
- 3\, \theta \, \bar u \,  u\, u_x
- \frac{3}{2}\, \theta  \, \bar u^2 \, u_x - \frac{3}{2}\, \theta  \, \bar u_x \, u^2
- 3 \, \theta \, \bar u_x \bar u \, u  - \theta\, \bar u \, u_t.
\]
Using
\[\bar u_x = \frac{\theta}{2} \,\frac{1}{a^2 - \bar u^2} \, \bar u^3\]

this yields
\begin{eqnarray}
\label{2.27}
F & = &
\tilde F(u,u_x,u_t) - \theta^2 \frac{3a^4\bar u^4-2a^2\bar u^6+\bar u^8}{2(a^2-\bar u^2)^3} \,u
-
\theta \,\frac{a^2\bar u}{a^2-\bar u^2}\, u_t
- \theta \, \frac{\bar u^4+3a^2\bar u^2}{2(a^2-\bar u^2)}\, u_x
\nonumber\\
& - & \theta^2 \, \frac{2\bar u^7-3a^2\bar u^5+3a^4\bar u^3}{4(a^2-\bar u^2)^3}\,u^2
-
2\, \bar u \, u_x^2
-\theta \, \frac{3a^2\bar u-\bar u^3}{a^2-\bar u^2}\, u \, u_x
\end{eqnarray}
with $\tilde F$ as defined in (\ref{2.22}).
%
%
%
For the second-order quasilinear hyperbolic equation (\ref{2.26}),
we consider the  initial conditions
\begin{align}\label{2.28}
t=0:\ u=\varphi(x),\ u_t=\psi(x),\ \ x\in[0,L]
\end{align}
and the boundary feedback conditions
\begin{align}\label{2.29}
&x=0:\ u_x=k\, u_t,
\\
\label{2.30} &x=L:\ u=0,
\end{align}
where $k>0$ is a real  constant.
We work in the framework of classical  semi-global solutions.
To apply the theory presented in
\cite{Wang ZQ1}, the second order equation is written
as a first order  system (see  the proof of Theorem 1 in \cite{LBW}).
In this way the following result can be obtained
(see  Lemma 1 in \cite{LBW}):

\begin{lemma}\label{well-posedness}
Let a subsonic stationary state $\bar u(x) >0$
as in Lemma \ref{lemmastationary}
be given.
Choose $T>0$ arbitrarily large.

There exist  constants
$\varepsilon_0(T)>0$
and $C_T>0$,
such that if the
initial data $(\varphi(x),\psi(x))\in C^2([0,L])\times C^1([0,L])$
satisfies
\begin{equation}
\label{epsilon0bedingung}
\max\left\{  \|\varphi(x)\|_{ C^2([0,L])},\;
   \|\psi(x))\|_{ C^1([0,L])}
   \right\}
   \leq \varepsilon_0(T)
\end{equation}
and the $C^2$-compatibility conditions are satisfied at the points
$(t,x)=(0,0)$ and  $(0,L)$,
then the initial-boundary problem
(\ref{2.26}),(\ref{2.28}),(\ref{2.29})-(\ref{2.30}) has a unique
solution $u(t,x)\in
C^2([0,T]\times [0,L])$.
Moreover the following estimate holds:
\begin{equation}
\label{apriori}
\| u\|_{ C^2([0,T]\times [0,L])} \leq C_T \max\left\{  \|\varphi(x)\|_{ C^2([0,L])},\;
   \|\psi(x))\|_{ C^1([0,L])}
   \right\}.
   \end{equation}
\end{lemma}


\section{Exponential stability}
In this section, we introduce a strict $H^2$-Lyapunov function
for the closed-loop system
consisting of the quasilinear wave equation (\ref{2.26})
and the boundary conditions (\ref{2.29}), (\ref{2.30}).
To motivate the choice of the Lyapunov function,
let us reconsider the
classical energy for
systems governed by the linear wave equation
$u_{tt} - c^2\, u_{xx}=0$, which
is $\int_0^L  c^2 \, (u_x)^2 + (u_t)^2  \, dx$.
In our quasilinear wave equation (\ref{2.26}),
instead of the square of the wave speed $c^2$ the term
$(a^2 - (\bar u + u)^2)$ appears
as a factor in front of $u_{xx}$, so it makes sense to
replace $c^2$ by this expression in the definition of  our
 Lyapunov  function.
In the same line of reasoning,
if our quasilinear equation would be
$$u_{tt} - (a^2 - (\bar u + u)^2) \, u_{xx}=F,$$
the integral
$\int_0^L   (a^2 - (\bar u + u)^2) \, (u_x)^2 + (u_t)^2  \, dx$
would be a  candidate for a Lyapunov function.
 However, in our wave equation also the term  $2(\bar u + u) u_{tx}$ appears.
In order to deal with this term, we  introduce an additional quantity
in our Lyapunov function in such a way that,
via  equation  (\ref{2.26}),
we can find an upper bound
for its time-derivative.
For this purpose, it makes sense to introduce a
term
that contains the product $u_t \, u_x$ in the integral defining the first part
of our Lyapunov function.
As a further motivation, we return to the
linear wave equation $u_{tt}-u_{xx}=0$
with the associated boundary conditions $u(t,0)=0$
and $u_t(t,L) = - k u_x(t,L)$ with $k>0$.
For a number $\lambda\in (0, \,
\frac{ 2 k}{L(1+k^2)})$,
the quantity
\[
E(t)= \int_0^L   (u_x)^2 + (u_t)^2  \,
+ \lambda \, x  \, u_x \, u_t  dx\]
can be used  to  show the exponential decay,
 since
 $E'(t) \leq -  \lambda \, (1 - \lambda L) \, E(t)$.

For many hyperbolic systems
 exponential weights in the Lyapunov function have been used
successfully,
see various examples in \cite{Coronbook}.
We define the
weights
\begin{eqnarray}
\label{h1defi}
h_1(x) & = & |k|, 
\\
\label{h2defi}
h_2(x) & = & \exp\left(- \frac{x}{L} \right).       
\end{eqnarray}
%
In the sequel we consider
\[
E_1(t)
=
\int_0^L
h_1(x) \,  \Big((a^2-(\bar u+u)^2)\, u_x^2+u_t^2\Big)-2  \, h_2(x)\,
\Big((\bar u+u)u_x^2+u_t \, u_x\Big)\,dx
\]
since
according to the previous considerations,
this is a natural candidate   to define a Lyapunov function for our system.

To show the exponential decay with respect to the
$H^2$-norm,
it is necessary  to deal with the second order derivatives.
Therefore we also introduce
$E_2(t)$ which is defined analogously to $E_1$ to
show the decay of the partial derivatives of second order.
We define
\[
E_2(t)=
\int_0^L h_1(x)\Big((a^2-(\bar
u+u)^2)u_{xx}^2+u_{tx}^2\Big)- 2 \, h_2(x) \,
\Big((\bar
u+u)u_{xx}^2+u_{tx}u_{xx}\Big)\,dx.
\]

We define the Lyapunov function $E(t)$ as
\begin{align}\label{4.1}
E(t)=  E_1(t)+E_2(t).
\end{align}

In the following subsection we show that
our  Lyapunov function $E(t)$ as defined in  (\ref{4.1})
is   bounded above and below by the product of
appropriate constants and the square of the $H^2$-norm of $u$.


\subsection{Equivalence of $\sqrt{E(t)}$
 with $E(t)$  as in (\ref{4.1})
and the $H^2$-norm of the state}
\label{equivalencesection}

%
%

In this section we show that
$\sqrt{E(t)}$
 with $E(t)$  as in (\ref{4.1}) is
equivalent to the $H^2$-norm of the state. This is a
an essential property of
a Lyapunov  function since we want to use it
to show the exponential decay of the $H^2$-norm.
%
%
Note that the constants in Lemma \ref{neulemmaneueversion}
are
independent of the length
$T$ of the time interval.





\begin{lemma}
\label{neulemmaneueversion}
Let a real number
\begin{equation}
\label{gammaintervalldefinition}
\gamma \in (0, \frac{1}{2}]
\end{equation}
 be given.
Choose a real number $k>0$ such that
\begin{equation}
\label{c2c1assumption}
\frac{1}{k} \in  (0,\, (1 - \gamma) \, a) .
\end{equation}
Assume that  $\bar u$ is
such that
we have
\begin{equation}
\label{ubarvor1}
\bar u \in (0, \,  \gamma  \, a).
\end{equation}

Then for the
weights
defined in (\ref{h1defi}), (\ref{h2defi})
on the interval $[0,L]$
we have  the strict inequality
\begin{eqnarray}
\label{41a}
h_2 &  < &  (a - \bar u) \, h_1.
\end{eqnarray}

In addition, we assume that $\bar u$ is sufficiently small in the sense that
\begin{equation}
\label{ubarvor2}
\sup_{x\in [0,L]} \frac{\bar u \, (a^2 - \bar u^2)}{a^2 + 3 \, \bar u^2} < \frac{1}{2\,{\rm e}\, k}.
\end{equation}
Then for the
weights we have the inequality
\begin{equation}
\label{41c}
h_2 >  \frac{\bar u (a^2 - \bar u^2)}{a^2 + 3 \bar u^2} 2\,
 h_1.
\end{equation}

For a real number $z$ define
\begin{equation}
\label{b11definition}
b_{11}(z) = 1 + 2\, z\, k - \frac{(1 + 2\, z\, k)^2}{k^2\,(a^2 - z^2)}.
\end{equation}
Assume that
\begin{equation}
\label{upsilonvorraussetzung}
\upsilon >k^2.
\end{equation}
Define the matrix
\begin{equation}
\label{tildeb3definition}
\tilde B_3(z)=
\left(
\begin{array}{ll}
b_{11}(z) &
\frac{1 + 2\, z\, k}{k\,(a^2 - z^2)} - k
\\
\frac{1 + 2\, z\, k}{k\,(a^2 - z^2)} - k & \upsilon - \frac{1}{a^2 - z^2}
\end{array}
\right).
\end{equation}
For a real number $z$ define
\begin{equation}
\label{cgdefinition}
 C_g(z)
=
 \frac{ a^2 - z^2}{{\rm e}\, z^2\left( 2
+ \frac{3}{2} \, \theta \, z +  \frac{ 2\,\theta\,z^3 }{  a^2 - z^2 } \right)^2}
\end{equation}
Define the matrix
 \begin{equation}
 \label{tildea3definition}
 \tilde A_3(z)
 =
 \left(
 \begin{array}{cc}
 \frac{1}{\rm e}\, \frac{a^2 + 3 z^2 }{a^2 - z^2} -  2\, k \,z
  &
  k - \frac{1}{\rm e}\, \frac{2\, z }{a^2 - z^2}
 \\
  k - \frac{1}{\rm e}\, \frac{2\, z }{a^2 - z^2}
 &
 C_g(z) + \frac{1}{\rm e}\, \frac{ 1 }{a^2 - z^2}
 \end{array}
 \right).
 \end{equation}
Then there exists $\varepsilon_1(\upsilon)>0$
such that for all $z$ with $|z| \leq
2 \,\varepsilon_1(\upsilon)$
the matrix $\tilde B_3(z)$ is positive definite
and the matrix $\tilde A_3(z)$ is positive definite.
\end{lemma}
\begin{proof}
First we show that (\ref{41a}) holds.
This is equivalent to the inequality
$$\frac{1}{k}
< \inf_{x\in [0,L]} \exp( \frac{1}{L}\, {x})
(a - \bar u(x)).
$$
Our assumptions (\ref{c2c1assumption}) and (\ref{ubarvor1}) imply that
$$\frac{1}{k} < (1 - \gamma) \, a
\leq
\inf_{x\in [0,L]} \, (a - \bar u(x)) \leq \inf_{x\in [0,L]} \exp\left( \frac{1}{L}\, x\right)
(a - \bar u(x)),$$
and (\ref{41a}) follows.
%
%
%
%
%
%
If  (\ref{ubarvor2}) holds,
we have
$$\left(\frac{h_2}{h_1}\right)
\geq
\frac{1}{{\rm e}\, k}
>  \sup_{x\in [0,\,L]} \frac{2\,\bar u \, (a^2 - \bar u^2)}{a^2 + 3 \bar u^2}
$$ and (\ref{41c}) follows.

Now we come to the assertion for the symmetric matrix $\tilde B_3$.
Due to (\ref{c2c1assumption}) we have
$b_{11}(0)= 1 - \frac{1}{k^2 \,a^2}>0$.
 Due to the continuity of $b_{11}( \cdot)$
 this implies that there exists a constant
 $\varepsilon_1>0$ such that for all
 $|z| \leq
 2  \,\varepsilon_1$ we have
 $b_{11}(z) >0$.
We have
\[
\det \tilde B_3(0) =
\det
\left(
\begin{array}{ll}
 1 - \frac{1}{k^2\, a^2}  & \frac{1}{k\, a^2} - k
 \\
  \frac{1}{k\, a^2} - k  &   \upsilon - \frac{1}{a^2}
\end{array}
\right)= \left( 1 - \frac{1}{k^2\, a^2} \right) \,( \upsilon - k^2).
\]
Hence (\ref{upsilonvorraussetzung}) implies $\det \tilde B_3(0)>0$.
 Due to the continuity of $\det \tilde B_3( \cdot)$
 this implies that we can choose the constant
 $\varepsilon_1>0$ in such a way that for all
 $|z| \leq
 2 \,
 \varepsilon_1$ we have
 $\det \tilde B_3(z) >0$,
 and thus $\tilde B_3(z)$ is positive definite.
We can choose the constant
 $\varepsilon_1>0$ in such a way that for all
 $|z| \leq
 2 \,
 \varepsilon_1$
 for the $2\times 2$ matrix $\tilde A_3(z)$
the upper left element in the matrix
is greater than zero.
We have
 \begin{equation}
 \label{cgblowup}
 \lim_{z\rightarrow 0+} C_g(z)=\infty.
 \end{equation}
Due to (\ref{cgblowup}) we can assume that
$\varepsilon_1>0$  is sufficiently small
 such  that for all
 $|z| \leq 2 \,\varepsilon_1$
we have  $\det \tilde A_3(z) >0$,
and thus $\tilde A_3(z)$ is positive definite.
\end{proof}



In Lemma \ref{normlemma} we show several inequalities that we need to show that
$E(t)$  as in (\ref{4.1}) can be bounded above and below by
the squared $H^2$-norm.
Note that also in Lemma \ref{normlemma}
the constants
are  independent of the length
$T$ of the time interval.

\begin{lemma}
\label{normlemma}

Let all assumptions of Lemma \ref{neulemmaneueversion} hold.
In particular, let
$k$ such that
(\ref{c2c1assumption}) holds be  given.
Let a stationary subsonic state $\bar
u(x)\in C^2(0,L)$ be given. Assume that  $\bar u$ is sufficiently
small in the sense that (\ref{ubarvor1}) and (\ref{ubarvor2})  hold.
%

For $x\in [0,L]$ and real numbers $v_0$  define the real function
\begin{equation}
\label{k1definition}
k_1(x,v_0)=
h_1(x)\, \left(a^2-(\bar u(x) + v_0)^2\right)-2 \, h_2(x)\,(\bar u(x) + v_0).
\end{equation}

If $\varepsilon_2>0$ is chosen sufficiently small, we have
\begin{align}
\label{5.44g}
&K  _1:=\min_{x\in[0,L]}\min_{|v_0|\leq\varepsilon_2}
k_1(x,v_0)-\frac{h_2^2(x)}{h_1(x)}>0,\\
\label{5.45g}
&\widetilde{K}_1:=\min_{x\in[0,L]}\min_{|v_0|\leq\varepsilon_2} \frac{h_1(x)  \, k_1(x,v_0)-h_2^2(x)}{k_1(x,v_0)}>0.
\end{align}

Assume in the sequel that $\varepsilon_2>0$ is chosen such that (\ref{5.44g}) and  (\ref{5.45g}) hold.

For $x\in [0,L]$ and real numbers $v_0$, $v_1$, $v_2$ define the real function
\begin{eqnarray}
&&\chi^x(v_0,v_1,v_2)=
\label{fxdefinition}
\\
&&h_1(x) \, \Big((a^2-(\bar u(x)+v_0)^2)v_1^2+v_2^2\Big)-
2 \, h_2(x) \,
\Big((\bar u(x) + v_0) v_1^2+ v_1 v_2\Big).
\end{eqnarray}
Then $\chi^x$ can be represented in the form
\begin{eqnarray}
&&\chi^x(v_0,v_1,v_2)
\\
\label{v1quadrat}
&&=\left(k_1(x,v_0)-\frac{h_2^2(x)}{h_1(x)}\right) v_1^2+\bigg(\sqrt{h_1(x) } \, v_2 -\frac{h_2(x)}{\sqrt{h_1(x)}} \, v_1\bigg)^2
\\
&&=
\label{v2quadrat}
\frac{h_1(x) \, k_1(x,v_0)-h_2^2(x)}{k_1(x,v_0)}\, v_2^2+\frac{1}{k_1(x,v_0)}\bigg(k_1(x,v_0) \, v_1 -h_2(x)  \, v_2\bigg)^2.
\end{eqnarray}
\end{lemma}
\begin{proof}

Inequality (\ref{41a}) implies that, if $\varepsilon_2>0$ is
sufficiently small and $|v_0|<\varepsilon_2$ we have
\begin{eqnarray*}
k_1(x,v_0) & = &
h_1(x) \, [a -(\bar u(x) + v_0)]\;  [a  + (\bar u(x) + v_0)] -2 \, h_2(x) \; [\bar u(x) + v_0]
\\
& > & ( h_2(x) - h_1(x) \, v_0) [a + \bar u(x) + v_0]
-2 \, h_2(x)  \; [\bar u(x) + v_0]
\\
& = & h_2(x) \, [a - \bar u(x)] -  [ (a + \bar u(x)) \, h_1(x) + h_2(x)] \, v_0  - h_1(x) \, v_0^2
\\
& > &
\frac{h_2^2(x)}{h_1(x)}
\end{eqnarray*}
which implies (\ref{5.44g}). This in turn implies  (\ref{5.45g}).


The representations (\ref{v1quadrat}) and (\ref{v2quadrat}) follow directly
from the definition of $\chi^x$.
%
%
%
%
%
%
%
\end{proof}

In the sequel we assume that the assumptions
from Lemma \ref{neulemmaneueversion}
hold.
With
$\chi^x$  as defined in Lemma  \ref{normlemma} we have
\begin{align}
\label{4.3} &
E_1(t)= \int_0^L \chi^x(u(t,x),  \, u_x(t,x), \, u_t(t,x))\,dx,
\\
\label{4.4} &E_2(t)=
\int_0^L \chi^x(u(t,x), \,  u_{xx}(t,x), \, u_{tx}(t,x))\,dx.
\end{align}


With these representations of $E_1$ and $E_2$,
Lemma \ref{normlemma} yields lower and upper bounds for $E_1(t)$ and $E_2(t)$.

\begin{lemma}
\label{h2lemma}
Assume that
\begin{equation}
\label{maxu} \max_{x\in [0,L]}  |u(t,x)|\leq \varepsilon_2
\end{equation}
where $\varepsilon_2$ is chosen as in Lemma \ref{normlemma}.
For $E_1$ defined in (\ref{4.3})
and $k_1$ defined in (\ref{k1definition})  we have
the lower bounds
\begin{align}\label{5.6}
E_1(t)\geq\int_0^L
\left(k_1(x,u(x))-\frac{h_2^2(x)}{h_1(x)}\right)u_x^2(x)\,dx\geq
K_1\int_0^L u_x^2\,dx,
\end{align}
and
\begin{align}\label{5.7}
E_1(t)\geq\int_0^L\frac{h_1(x) k_1(x,u) -
h_2^2(x)}{k_1(x,u)} \, u_t^2(x)\,dx\geq \widetilde{K}_1\int_0^L
u_t^2\,dx.
\end{align}
Moreover, we have the upper bounds
\begin{equation}
\label{5.38} E_1(t)
\leq\int_0^L\left(k_1(x,u(t,x))+\frac{h_2^2(x)}{h_1(x)}\right)u_x^2(t,x)+2 \, h_1(x) \,
u_t^2(t,x)\, dx
\end{equation}
and
\begin{equation}
\label{5.38a}
E_1(t) \leq \int_0^L 2 \, h_1(x)\, \left(a^2 - (\bar u + u)^2 \right) u_x^2(t,x) + 2 \, h_1(x) \, u_t^2(t,x) \, dx.
\end{equation}

For $E_2$ defined in (\ref{4.4}), we have
the lower bounds
\begin{align}\label{5.9}
E_2(t)\geq\int_0^L
\left(k_1(x,u(x))-\frac{h_2^2(x)}{h_1(x)}\right)u_{xx}^2(x)\, dx\geq
K_1\int_0^L u_{xx}^2\,dx,
\end{align}
and
\begin{align}\label{5.10}
E_2(t)\geq\int_0^L\frac{h_1(x) k_1(x,u) -
h_2^2(x)}{k_1(x,u)}u_{tx}^2(x)\,dx\geq \widetilde{K}_1\int_0^L
u_{tx}^2\,dx.
\end{align}
Moreover, we have the upper bounds
\begin{equation}
\label{5.39} E_2(t)
\leq\int_0^L\left(k_1(x,u(t,x))+\frac{h_2^2(x)}{h_1(x)}\right)u_{xx}^2(t,x)+2 \, h_1(x) \,
u_{tx}^2(t,x)\, dx.
\end{equation}
\begin{equation}
\label{5.38agugat}
E_2(t) \leq \int_0^L 2 h_1(x)\, \left(a^2 - (\bar u + u)^2 \right) u_{xx}^2(t,x) + 2 \, h_1(x) \, u_{tx}^2(t,x) \, dx.
\end{equation}
\end{lemma}
\begin{proof}
%
%

Equation (\ref{v1quadrat}) and (\ref{5.44g}) imply the lower bound (\ref{5.6}) for $E_1$.

The representation  (\ref{v2quadrat}) and (\ref{5.45g}) imply the lower bound (\ref{5.7}).

The upper bound (\ref{5.38}) follows from (\ref{v1quadrat}) and  Young's inequality.
The upper bound (\ref{5.38a})  follows from   (\ref{5.38})
using
$\frac{h_2^2}{h_1} < k_1$
 and the definition of $k_1$.

The representations (\ref{v1quadrat}) and (\ref{5.44g}) also imply the lower bound (\ref{5.9}) for $E_2$,
and  (\ref{v2quadrat}) and (\ref{5.45g}) imply the lower bound (\ref{5.10}).
The upper bound (\ref{5.39}) again follows from (\ref{v1quadrat}) and  Young's inequality.
The upper bound (\ref{5.38agugat})  follows from   (\ref{5.39})
using
$\frac{h_2^2}{h_1} < k_1$
and the definition of $k_1$.
\end{proof}



Now we can show that
$E(t)$ can be bounded above and below by
the squared $H^2$-norm.
Define the number
\begin{equation}
\label{kmaxdefinition}
K_{\max} =
\max \left\{
2\,k,\; \max_{x\in[0,L]}\max_{|v_0|\leq\varepsilon_2}
k_1(x,v_0) + \frac{h_2^2(x)}{h_1(x)}
\right\}.
\end{equation}
If (\ref{maxu}) holds, by (\ref{5.38}) and (\ref{5.39}) Lemma \ref{h2lemma}
implies the inequality
\begin{equation}
\label{ober}
E(t)
\leq
K_{\max} \int_0^L u^2(t,x) + u^2_x(t,x) + u_t^2(t,x) +  u_{tx}^2(t,x) +  u_{xx}^2(t,x)\, dx.
\end{equation}
Define
\begin{equation}
\label{kmindefinition}
K_{\min} = \frac{1}{2} \;\min \{ K_1, \, \tilde K_1\}.
\end{equation}
By the definition of $E$ and (\ref{5.6}), (\ref{5.7}), (\ref{5.9}), (\ref{5.10})
we also have the lower bound
\begin{equation}
\label{unter}
E(t) \geq K_{\min }
\int_0^L  u^2_x(t,x) + u_t^2(t,x) +  u_{tx}^2(t,x) +  u_{xx}^2(t,x)\, dx.
\end{equation}
The Poincar\'e-Friedrichs inequality states that if (\ref{2.30}) holds,
we have
\begin{equation}
\label{poin} \int_0^L u^2(t,x)\, dx \leq 2 \, L^2 \, \int_0^L
u_x^2(t,x)\,dx.
\end{equation}
Using this inequality and (\ref{2.26}),
inequality  (\ref{unter}) implies that if
$E(t)$ is small, also the $H^2$-norm of $u(t,x)$ is small.
Similarly $E_1(t)$ can be bounded above and below by
the squared $H^1$-norm.

\subsection{Exponential Decay of the $H^2$-Lyapunov Function}
\label{sec4}

In this section we present our main result about the exponential decay of the Lyapunov function
that we have introduced in  (\ref{4.1}).
%
Consider the system
\begin{equation}
\label{revisionsystem}
\left\{
\begin{array}{l}
\tilde u_{tt} + 2 \, \tilde u \, \tilde u_{tx}-(a^2-\tilde u^2)
\, \tilde
u_{xx}=\tilde F(\tilde u,\tilde u_x,\tilde u_t),
\\
\tilde u_x(t, \, 0) =\bar u_x(0)+k \, \tilde u_t(t,\,0),\, t\in [0,T],
\\
\tilde u(t,\, L) =\bar u(L), \, t\in [0,T],
\\
t=0:\ \tilde u=\varphi(x) + \bar u(x),\ \tilde u_t=\psi(x),\ \ x\in[0,L]
\end{array}
\right.
\end{equation}
with $\tilde F$ as defined in (\ref{2.22}).
In Theorem \ref{thm4.2}
we present our main result about the stabilization of
(\ref{revisionsystem})
for $\tilde u$.
For the analysis we use the fact that (\ref{revisionsystem})
is equivalent to
 (\ref{2.26}),(\ref{2.28}),(\ref{2.29}),(\ref{2.30})
 that is stated
 in terms of $u$ which is defined
in  (\ref{udefinition2016}) as the difference between
$\tilde u$ and the stationary state $\bar u$.
In Theorem \ref{thm4.2} we state that
the function
$E(t)$  defined in (\ref{4.1}) is a strict Lyapunov function.
In  Theorem \ref{thm4.2} it is assumed
that $\bar u>0$ is sufficiently small and
$k$ is sufficiently large.
Before we state the theorem, in the following remark we
comment on  condition
(\ref{kpartialassumption})
that appears in the statement of the Theorem
and explain why it can be satisfied for all $a>0$
if $\bar u>0$ is sufficiently small and
$k$ is sufficiently large.

\begin{rem}
\label{voraussetzungserklaerung}
For $k>0$ and $\bar u_0>0$ define
\begin{equation}
\label{upsilondefinition}
K_{\partial}(k,  \,\bar u_0 )=
2\,
\left[
\frac{4}{k^2} +  \frac{2 \, \bar u_0}{k} +
\theta \, \frac{\bar u_0^4+3a^2\bar u_0^2 + \frac{2}{k} a^2 \bar u_0 }{2(a^2-\bar u_0^2)}
+
 \frac{5}{2}\,\frac{\theta}{k^2}   +
\frac{\theta}{k} \, \frac{3a^2\bar u_0-\bar u_0^3}{a^2-\bar u_0^2}
\right]^2.
\end{equation}
Then (\ref{upsilondefinition}) implies
\[
\lim_{\bar u_0 \rightarrow 0+}
K_{\partial}(k,  \,\bar u_0 )
=
\lim_{\bar u_0 \rightarrow 0+}
2\,
\left[
\frac{4}{k^2}
+  \frac{5}{2}\,\frac{\theta}{k^2}
\right]^2
 =
 \frac{2}{k^4}
\left[
4  +  \frac{5}{2}\, \theta \right]^2
.
\]
This in turn implies that
\[\lim_{\bar u_0 \rightarrow 0+}
2\, k^2\,K_{\partial}(k,  \,\bar u_0 )
-
\left[ a^2 - \left(\bar u_0 + \frac{2}{k} \right)^2 \right]
=
\frac{4}{k^2}\,
\left[
4   +  \frac{5}{2}\, \theta \right]^2 - a^2  + \frac{4}{k^2}
.\]
Hence we have
\[
\lim_{k\rightarrow \infty}
\lim_{\bar u_0 \rightarrow 0+}
2\, k^2\,K_{\partial}(k,  \,\bar u_0 )
-
\left[ a^2 - \left(\bar u_0 + \frac{2}{k} \right)^2 \right]
=
-a^2 <0.\]
This implies
that if $\bar u_0>0$ is sufficiently small and
$k$ is sufficiently large,
then
condition (\ref{kpartialassumption})
with $\bar u(0)= \bar u_0$
in  Theorem \ref{thm4.2} below
holds.
In fact, if ${\bar u_0}>0$ is sufficiently small, for
 $k= \frac{1}{\bar u_0}$ condition (\ref{kpartialassumption}) holds,
since
\[
\lim_{k\rightarrow \infty}
2\, k^2\,K_{\partial}\left(k,  \, \frac{1}{k} \right)
-
\left[ a^2 - \left( \frac{1}{k} + \frac{2}{k} \right)^2 \right]
= - a^2.
\]


\end{rem}

\begin{thm}
\label{thm4.2}
{\bf (Exponential Decay of the $H^2$-Lyapunov Function).}\label{thm1}
Let a real number $\gamma \in (0, \frac{1}{2}]$ be given.
%
%
%
%
%
%
%
%
%
Choose a real number $k>0$ such that
\begin{equation}
\label{c2c1assumptionneu}
\frac{1}{a\,k}
< 1 - \gamma.
\end{equation}
Let a stationary subsonic state $\bar u(x)\in C^2(0,L)$ be given
that satisfies (\ref{statvor}). 
Assume that for all $x\in L$ we have
$
   \bar u(x) \in\left(0, \,  \gamma \, a\right).
 $ 
%
Assume that
for $K_{\partial}(k,  \,\bar u_0 )$  as defined in (\ref{upsilondefinition})
we have
\begin{equation}
\label{kpartialassumption}
2 \, k^2 \, K_{\partial}(k,  \,\bar u( 0) )
\leq a^2 - \left(\bar u(0) + \frac{2}{k} \right)^2.
\end{equation}
Assume that
$\|\bar u\|_{C^2([0,L])}$ is sufficiently small
such that
$\|\bar u\|_{C([0,L])} < \varepsilon_1(2\, k^2)$
(with
$\varepsilon_1$ from Lemma \ref{neulemmaneueversion})
 and
(\ref{ubarvor2}) holds.

Let $T>0$ be given.
%
%
%
%
If  the initial data satisfies
\begin{align}\label{4.8}
\|(\varphi(x),\psi(x))\|_{C^2([0,L])\times C^1([0,L])}\leq\varepsilon_0(T)
\end{align}
and the
$C^2$-compatibility conditions at the points
$(t,x) = (0,0)$ and $(t, x)=(0,L)$.
the initial-boundary value problem
(\ref{revisionsystem}) for $\tilde u$
%
has  a unique classical solution
$\tilde u \in C^2([0,T]\times [0,L])$.
Problem (\ref{2.26}),(\ref{2.28}),(\ref{2.29}),(\ref{2.30})
has  a unique classical solution $u\in C^2([0,T]\times [0,L])$
that satisfies
the a priori estimate (\ref{apriori}).
Since
$\varepsilon_1(\upsilon)$ from Lemma \ref{neulemmaneueversion}
with  $\upsilon=2 \,k^2$
and $\varepsilon_2$ from Lemma
\ref{normlemma}
are independent of $T$,
we can choose the constant $\varepsilon_0(T) > 0$
from Lemma \ref{well-posedness}
sufficiently small
such that
the a priori estimate (\ref{apriori})
implies that
\begin{equation}
\label{uistklein}
\|u\|_{C([0,T]\times [0,L])} \leq
\min\{ \varepsilon_1( 2 \,k^2 ),\,  \varepsilon_2  \}
\end{equation}
and for all $t\in [0,T]$, $x\in [0,L]$ we have the inequalities
  \begin{equation}
  \label{i3vor}
  |u(t,\,x) | \leq
  \min \left\{ \gamma \, a - |\bar u(x)|,\,|\bar u(x)|, \,\frac{1}{k}\right\}\,
  \;
\mbox{\it and }\;
  |u_x(t,x) |\leq \min\left\{ 1, \, \frac{1}{k}\right\}.
  \end{equation}

Moreover, choose
$\varepsilon_0(T)$ and $\bar u$  sufficiently small  such that
\begin{eqnarray}
\nonumber
\kappa  & := &
\max_{t \in[0,T] }
(P_0(T_{Li}(t))+P_1(T_{Li}(t)))(1+L^2)\left(\frac{1}{K_1}+\frac{1}{\tilde
K_1}\right)
 + 2 \, k^2 \, C_{E1}(\bar u(0) ) \, \frac{  L}{K_1}
 \\
 \label{kappadefinition}
 &
\leq
&
 \frac{1}{4\,{\rm e} \, L\,  k }
\end{eqnarray}
with the functions
$P_0$, $P_1$ and  $C_{E1}$
defined  in (\ref{P_0}), (\ref{P_1}),  (\ref{ce1definition})
and
\begin{equation}
\label{T(t)}
T_{Li}(t)= \max_{x\in [0,L]} \left\{|u(t,x)|, |u_x(t,x)|,
|u_t(t,x)|,\, |\bar u(x)|, |\bar u'(x)|\right\}.
\end{equation}
Define the number
\begin{equation}
\label{mudefinition} \mu =   \frac{1}{2\,{\rm e} \, L\,  k } - \kappa
\geq \frac{1}{4 \,{\rm e} L k}.
\end{equation}
Then we have
\begin{equation}
\label{e1lyapunov}
E_1(t) \leq E_1(0) \, \exp\left(-
\mu \, t\right)\, \, \mbox{\rm for all }\; t\in [0, T],
\end{equation}
\begin{align}\label{4.9}
E(t)\leq E(0)  \, \exp\left(-
\mu \, t\right)\, \, \mbox{\rm for all }\; t\in [0, T]
\end{align}
that is $E_1(t)$ and $E(t)$ as defined in (\ref{4.1})
are
strict Lyapunov functions for our control system (\ref{revisionsystem}).
%
\end{thm}

\begin{rem}
Theorem \ref{thm4.2}
states that
if $\bar u>0$ is sufficiently small and
$k$ is sufficiently large
for sufficiently small initial data
the Lyapunov function
decays exponentially and the decay rate
is at least
$\mu_0 = \frac{1}{4 \,{\rm e} \, L \, k}$
which is independent of $\bar u$ and  $T$,
since  the conditions on $k$ do not depend on
$T$.
For
arbitrarily large $T$,
we can always achieve this decay rate
$\mu_0$  for sufficiently small initial
data.
With this decay rate,
it is possible to determine
a time $T_0>0$ when the size of the $H^2 \times  H^1$--norm
of the solution is reduced at least by a factor 1/3.
In fact let
\begin{equation}
\label{Tschranke0305}
T_0 =   4 \,{\rm e} \,L \, k {\,{\rm ln}\left(  9 \,(1+2 \, L^2) \; \frac{K_{\max} }{K_{\min}}  \right)}
\end{equation}
with $K_{\max}$ from (\ref{kmaxdefinition}) and $K_{\min}$ from (\ref{kmindefinition}).
Then due to (\ref{unter}) and (\ref{poin}) we have
\begin{eqnarray*}
& & \|(u(T_0,\cdot), \, u_t(T_0,\cdot))\|_{H^2(0,L)\times H^1(0,L)}^2
\\
& = &
\int_0^L u^2(T_0,x) + u^2_x(T_0,x) +  u_{xx}^2(T_0,x) + u_t^2(T_0,x) +  u_{tx}^2(T_0,x) \, dx
\\
& \leq &
\frac{1 + 2\, L^2}{K_{\min}}\, E(T_0).
\end{eqnarray*}
If the assuption of Theorem \ref{thm4.2} hold for
the time interval $[0,\, T_0]$, by
(\ref{4.9}) and (\ref{ober})  this yields
\begin{eqnarray*}
& & \|(u(T_0,\cdot), \, u_t(T_0,\cdot))\|_{H^2(0,L)\times H^1(0,L)}^2
\\
& \leq &
\frac{1 + 2\, L^2}{K_{\min}}\, E(0)  \, \exp\left(-
\mu_0 \, T_0\right)
\\
& \leq &
(1 + 2\, L^2)
\frac{K_{\max}}{K_{\min}} \,
 \, \exp\left(-
\mu_0 \, T_0\right)
\,
\|(\varphi(x),\psi(x))\|^2_{H^2(0,L)\times
H^1(0,L)}
\\
& = &
\frac{1}{9} \;
\|(\varphi(x),\psi(x))\|^2_{H^2(0,L)\times
H^1(0,L)}
\\
& = &
\frac{1}{9} \;
\|(u(0,\,\cdot),u_t(0,\,\cdot))\|^2_{H^2(0,L)\times
H^1(0,L)}
.
\end{eqnarray*}
\end{rem}

\begin{corollary}
\label{cor4.3}
{\bf (Exponential Decay of the $H^2$-Norm and the $C^1$-Norm)}\label{cor1}
Under the assumptions of Theorem \ref{thm4.2},
for the semi-global classical solution $u$
of the mixed initial-boundary value problem
(\ref{2.26}), (\ref{2.28}), (\ref{2.29}), (\ref{2.30}) the $H^2$-norm
decays exponentially with time on $[0,T]$. More precisely, there
exists a constant $\eta_1 > 0$
that is independent of $T$ such that for any $t\in[0, T]$ the
inequality
\begin{equation}\label{4.10}
\|(u(t,\cdot),u_t(t,\cdot))\|_{H^2(0,L)\times
H^1(0,L)}\leq\eta_1\|(\varphi(x),\psi(x))\|_{H^2(0,L)\times
H^1(0,L)}\exp\left(-\frac{\mu}2\, t\right)
\end{equation}
holds. Furthermore, there exists a constant $\eta_2>0$
that is independent of $T$
such that for
any $t\in[0,T]$ the $C^1$-norm of the solution satisfies
\begin{align}\label{4.11}
\|(u(t,\cdot),u_t(t,\cdot))\|_{C^1[0,L]\times
C^0[0,L]}\leq \eta_2\|(\varphi(x),\psi(x))\|_{H^2(0,L)\times
H^1(0,L)}\exp\left(-\frac{\mu}2 \,t\right).
\end{align}
Due to (\ref{mudefinition}), this implies
that for $T$ sufficiently large we have
\begin{equation}\label{4.1002.05a}
\|(u(T,\cdot),u_t(T,\cdot))\|_{H^2(0,L)\times
H^1(0,L)}\leq
\frac{1}{2}
\|(\varphi(x),\psi(x))\|_{H^2(0,L)\times
H^1(0,L)}
\end{equation}
\mbox{ and}
\begin{equation}\label{4.1002.05b}
\|(u(T,\cdot),u_t(T,\cdot))\|_{C^1[0,L]\times
C^0[0,L]}\leq
\frac{1}{2}
\|(\varphi(x),\psi(x))\|_{H^2(0,L)\times
H^1(0,L)}.
%
\end{equation}
\end{corollary}
The proofs of Theorem \ref{thm4.2} and Corollary \ref{cor4.3} are given in Section \ref{sec5}.


\section{Proofs of Theorem \ref{thm4.2} and Corollary \ref{cor4.3}}
\label{sec5}
In this section we prove Theorem \ref{thm4.2} and Corollary \ref{cor4.3} from Section
\ref{sec4}.
For the proof, we consider the time derivative of
the Laypunov function $E(t)$.


\subsection{Time derivative of the Lyapunov function}
First we consider the evaluation of the time derivative of the Lyapunov function $E(t)$.

\begin{lemma}
\label{lemmaE1}
Let the assumptions of Theorem \ref{thm4.2} hold.
Then the time-derivative of $E_1$ is given by the
following
 equation:
\begin{eqnarray*}
\frac{d}{dt}E_1(t)
& = & I_1 + I_2 + I_3
\end{eqnarray*}
with
\begin{eqnarray}
\label{5.15}
I_1 &= &
\int_0^L h_{2x} \, (a^2-(\bar u+u)^2) \,  u_x^2  +
 h_{2x} \, u_t^2 \,dx,
\\
\label{5.16}
I_2& = &
\int_0^L
2 \, h_1 (\bar u'+u_x) \, u_t^2
-2\, h_1\, (\bar u+u)u_t \, u_x^2+ 4 \, h_1(\bar u+u)(\bar u'+u_x) u_x  u_t
\nonumber
\\
& + & 2 \, h_1  \, F \, u_t
 -2 \, h_2 \, u_t \, u_x^2 - 2 \, h_2 \, (\bar u+u)(\bar u'+u_x) u_x^2-2 \, h_2 \, F \, u_x \,dx,
 \\
\label{5.17}
I_3 &= &
[(a^2-(\bar u+u)^2)(2 \, h_1 \, u_x \, u_t-h_2 \, u_x^2)-(2\, h_1\, (\bar u+u)+h_2) \,u_t^2]_{x=0}^L.
\end{eqnarray}

%

\end{lemma}
\begin{proof}


With the notation $\hat d = a^2-(\bar u+u)^2 $
we have
$\hat d_t = - 2 (\bar u+u) \, u_t$,

$\hat d_x = - 2 (\bar u+  u ) \,(\bar u'+  u_x ) $
 and
\[
E_1(t)=
\int_0^L
h_1 \Big(  \hat d\, u_x^2 +  u_t^2\Big)-2 \,  h_2 \,
\Big((\bar u+u)u_x^2+u_t \, u_x\Big)\,dx.
\]
Hence differentiation yields
\begin{eqnarray*}
\frac{d}{dt}E_1(t) & = &
\int_0^L
2\,h_1 \Big[
(u_{tt}
-
(\bar u + u) \, u_x^2 )
\, u_t +
\hat d\,  u_x\, u_{xt}
\Big]
\\
&
-
& 2\,  h_2 \,
\Big[u_t \, u_x^2 + (  u_{tt} + 2 (\bar u+u)\, u_{xt})u_x
 +   u_t\, u_{xt}  \Big]\,dx.
\end{eqnarray*}
Now integration  by parts
for the term  $\hat d\,  u_x\, u_{xt} = \left(\hat d\,  u_x \right)\, \left(u_t\right)_x$
yields the equation
\begin{eqnarray*}
\frac{d}{dt}E_1(t) & = &
\int_0^L
2\,h_1 \Big[u_{tt}
-
\hat d\,  u_{xx}
-
\hat d_x\,  u_x\,
-
(\bar u + u) \, u_x^2
\Big]\, u_t
\\
&
-
& 2\,  h_2 \,
\Big[u_t \, u_x^2 + (  u_{tt} + 2 (\bar u+u)\, u_{xt})u_x
 +   u_t\, u_{xt}  \Big]\,dx
 +
\left[ 2\,h_1\, \hat d\,  u_x\, u_{t}\right]_{x=0}^L.
\end{eqnarray*}
%
Hence we get  the equation
\begin{align*}
&\frac{d}{dt}E_1(t)
= \int_0^L 2 \, h_1
\left[(u_{tt}- \hat d\, u_{xx})\, u_t
-\hat d_x
\, u_x  \, u_t
-(\bar u+u)  u_x^2 u_t \right]
\\
&\ \ \ \ \ \ \
-2 \, h_2 \,
\left[ u_t u_x^2+(u_{tt}+2(\bar
u+u)u_{tx})u_x + u_t u_{tx}
\right]
\, dx
\\
&\ \ \ \ \ \ \
+
\left[2h_1 \, \hat d \, u_x \, u_t\right]_{x=0}^L.
\end{align*}
By the partial differential equation (\ref{2.26})
we
have
$u_{tt}- \hat d\, u_{xx}
=
F-2(\bar u+u)u_{tx}
$
and
obtain
\begin{align}\label{5.14}
\frac{d}{dt} E_1(t)
& = \int_0^L 2h_1
\left[(F-2(\bar u+u)u_{tx})u_t+2
(\bar u+u)(\bar u'+u_x) u_t u_x - (\bar u+u) u_t \, u_x^2
\right]\nonumber\\
&\ \ \ \ \ \ \
-2 h_2
\left[
(F+
\hat d\,
u_{xx})u_x + u_t \, u_x^2 + u_t  \, u_{tx}
\right] \, dx
+
\left[
2 \, h_1 \,
\hat d
\,
 u_x
 \,
  u_t
\right]_{x=0}^L.
\nonumber
\\
\end{align}
Using integration by parts we obtain the identities
\begin{eqnarray*}
\int_0^L -4h_1(\bar u+u)u_t u_{tx} \, dx & = &
\int_0^L - 2 h_1(\bar u+u) (u_t^2)_x \, dx
\\
& = & [-2h_1 u_t^2(\bar u+u)]_{x=0}^L +\int_0^L 2\, (h_1(\bar
u+u))_xu_t^2\, dx
\end{eqnarray*}
and
\begin{eqnarray*}
\int_0^L -2 \, h_2 \, \hat d\,u_x u_{xx} - 2 \, h_2 \, u_t \, u_{tx} \, dx  & = &
\int_0^L - h_2   \, \hat d\, (u_x^2)_x  - h_2 \, (u_t)^2_x\, dx
\end{eqnarray*}
\begin{eqnarray*}
& = &
[-h_2  \, \hat d\, u_x^2  - h_2 \, u_t^2]_{x=0}^L
 +
\int_0^L h_{2x}   \, \hat d\, u_x^2
- 2 \, h_2(\bar
u+u)(\bar u'+u_x) \, u_x^2 + h_{2x} \, (u_t)^2\, dx.
\end{eqnarray*}
Using these identities we obtain the equation
$\frac{d}{dt} E_1(t) =I_1 + I_2 + I_3$.
Here, $I_3$ contains all the terms coming from the boundary
and $I_1= \int_0^L h_{2x} \,\hat d\, u_x^2 + h_{2x} \,u_t^2 \, dx$
contains all the terms where $h_{2x}$ appears.
The remaining terms appear in $I_2$.
%
%
\end{proof}


Similarly the next lemma is proved, where the time derivative of $E_2$ is considered.
\begin{lemma}
\label{lemmaE2}
Let the assumption of Theorem \ref{thm4.2} hold.
Then the following equation holds:
\begin{eqnarray}
\label{5.23}
\frac{d}{dt}
E_2(t)
& = &
\tilde{I}_1+\tilde{I}_2+\tilde{I}_3
\end{eqnarray}
with
\begin{eqnarray}
\label{5.24} \tilde{I}_1& = &\int_0^L h_{2x}\,(a^2-(\bar u+u)^2)\,
 u_{xx}^2 +
h_{2x} \, u_{tx}^2  \, dx,
\\
\label{5.25} \tilde{I}_2 &= & \int_0^L
4\,  h_2 \,(\bar u'+ u_x)\, u_{xx} \, u_{tx}
-2 \,   h_2 \,  u_t \,  u_{xx}^2 + 2 \, h_2 \,
(\bar u+u)\, (\bar u'+u_x) \, u_{xx}^2 \nonumber\\
& - & 2 h_2 \, F_x  \,  u_{xx} -2  h_1 \, (\bar u+u) \, u_t
\, u_{xx}^2- 2  h_1 \, (\bar u'+u_x) \, u_{tx}^2+2  h_1 \, F_x \,  u_{tx}\,dx,
\\
\label{5.26} \tilde{I}_3
&= &
[(a^2-(\bar u+u)^2)(2 h_1 \, u_{xx} \, u_{tx}- h_2 \, u_{xx}^2)-
(2\, h_1 \,
(\bar u+u)+h_2) u_{tx}^2]_{x=0}^L.
\end{eqnarray}
\end{lemma}
\begin{proof}
Again using  the notation $\hat d = a^2-(\bar u+u)^2 $
we have
\begin{eqnarray*}
\frac{d}{dt}E_2(t) & = &
\int_0^L
2\,h_1 \Big[
\, u_{ttx} \,  u_{tx}
-
(\bar u + u) \, u_{xx}^2 \, u_t +
\hat d\,  u_{xx}\, u_{xxt}
\Big]
\\
&
-
& 2\,  h_2 \,
\Big[u_{t} \, u_{xx}^2 + (  u_{ttx} + 2 (\bar u+u)\, u_{xxt})\, u_{xx}
 +   u_{tx}\, u_{txx}  \Big]\,dx.
\end{eqnarray*}
Integration  by parts
for the term  $\hat d\,  u_{xx}\, u_{xxt} = \left(\hat d\,  u_{xx} \right)\, \left(u_{tx}\right)_x$
yields the equation
\begin{eqnarray*}
\frac{d}{dt}E_2(t) & = &
\int_0^L
2\,h_1 \Big[u_{ttx}\, u_{tx}
-
(\bar u + u) \, u_{xx}^2\,  u_t\,
-
\hat d\,  u_{xxx} \,u_{tx}
-
\hat d_x\,  u_{xx} \,u_{tx}
\Big]
dx
\\
&
+
&
\left[ 2\,h_1\, \hat d\,  u_{xx}\, u_{tx}\right]_{x=0}^L
\\
&
-
&\int_0^L 2\,  h_2 \,
\Big[u_{t} \, u_{xx}^2 + (  u_{ttx} + 2 (\bar u+u)\, u_{xxt})\, u_{xx}
 +   u_{tx}\, u_{txx}  \Big]
\,dx.
\end{eqnarray*}
Hence we get  the equation
\begin{eqnarray*}
\frac{d}{dt}E_2(t)
&=&
\int_0^L
2\,h_1 \Big[
(u_{ttx}
-
\hat d\,  u_{xxx})
\, u_{tx}
-
(\bar u + u) \, u_{xx}^2\,  u_t\,
-
\hat d_x\,  u_{xx} \,u_{tx}
\Big]
dx
\\
&
+
&
\left[ 2\,h_1\, \hat d\,  u_{xx}\, u_{tx}\right]_{x=0}^L
\\
&
-
&
\int_0^L 2\,  h_2 \,
\Big[u_{t} \, u_{xx}^2 + (  u_{ttx} + 2 (\bar u+u)\, u_{xxt})\, u_{xx}
 +   u_{tx}\, u_{txx}  \Big]
\,dx.
\end{eqnarray*}
By the partial differential equation (\ref{2.26})
we
have
\[
u_{ttx}- \hat d\, u_{xxx}
=
\hat d_x \, u_{xx}
+
F_x-
2(\bar u'+u_x)u_{tx}
-2(\bar u+u)u_{txx}
%
\]
and
obtain
\begin{eqnarray*}
\frac{d}{dt}E_2(t)
&=&
\int_0^L
2\,h_1 \Big[
\left(\hat d_x \, u_{xx}
+
F_x-
2(\bar u'+u_x)u_{tx}
-2(\bar u+u)u_{txx}
\right)
\, u_{tx}
\\
&
-
&
(\bar u + u) \, u_{xx}^2\,  u_t\,
-
\hat d_x\,  u_{xx} \,u_{tx}
\Big]
dx
+
\left[ 2\,h_1\, \hat d\,  u_{xx}\, u_{tx}\right]_{x=0}^L
\\
&
-
&
\int_0^L 2\,  h_2 \,
\Big[u_{t} \, u_{xx}^2 +
\left(
\hat d\, u_{xxx}+\hat d_x \, u_{xx} + F_x - 2\, (\bar u'+u_x)u_{tx}
\right)\, u_{xx}
 +   u_{tx}\, u_{txx}  \Big]
\,dx
\\
&=&
\int_0^L
2\,h_1 \Big[
\left(
F_x-
2(\bar u'+u_x)u_{tx}
-2(\bar u+u)u_{txx}
\right)
\, u_{tx}
-
(\bar u + u) \, u_{xx}^2\,  u_t\,
\Big]
dx
\\
&
+
&
\left[ 2\,h_1\, \hat d\,  u_{xx}\, u_{tx}\right]_{x=0}^L
\\
&
-
&
\int_0^L 2\,  h_2 \,
\Big[u_{t} \, u_{xx}^2 +
\left(
\hat d\, u_{xxx}+\hat d_x \, u_{xx} + F_x - 2\, (\bar u'+u_x)u_{tx}
\right)\, u_{xx}
 +   u_{tx}\, u_{txx}  \Big]
\,dx
.
\end{eqnarray*}
Using integration by parts we obtain the identities
\[
\int_0^L -2\, h_2 \left( \hat d\, u_{xxx}+\hat d_x \, u_{xx} \right)\, u_{xx}
 -2\, h_2\,  u_{tx}\, u_{txx}
 \,
dx
\]
\[
=
\int_0^L - h_2 \,  \hat d\, (u_{xx}^2)_x
-2\, h_2 \, \hat d_x \, (u_{xx})^2
 - h_2\,  (u_{tx}^2)_x
 \,
dx
\]
\begin{eqnarray*}
& = &
\int_0^L
h_{2x} \hat d \, u_{xx}^2 + h_{2x} u_{tx}^2
+
2 \, h_2 \, (\bar u + u) (\bar u' + u_x) \, u_{xx}^2
 \, dx
 \\
 & + &
 \left[ - h_2 \, \hat d \, u_{xx}^2 - h_2 \, u_{tx}^2 \right]_{x=0}^L
\end{eqnarray*}
and
\[
\int_0^L -4h_1(\bar u+u)u_{tx}\, u_{txx} \, dx
=
 [-2 \, h_1 \, (\bar u+u) \, u_{tx}^2]_{x=0}^L +\int_0^L 2\, h_1
 \, (\bar u'+u_x) \, u_{tx}^2\, dx.
\]

Using these identities we obtain
$\frac{d}{dt} E_2(t) = \tilde I_1 +  \tilde I_2 +  \tilde I_3$
where $\tilde I_3$ contains all the terms coming from the boundary
and $\tilde I_1= \int_0^L h_{2x} \,\hat d\, u_{xx}^2 + h_{2x} \,u_{tx}^2 \, dx$
contains all the terms where $h_{2x}$ appears.
\end{proof}


%
%


\subsection{Proof of Theorem \ref{thm4.2}}

\begin{proof}
In the proof, we use Lemma \ref{neulemmaneueversion}.
Therefore we assume that $\bar u$ is sufficiently small in the
sense that (\ref{ubarvor2}) holds.
Moreover, we use Lemma \ref{normlemma}.
Therefore we assume that $\varepsilon_0(T)>0$ is sufficiently small
such that (\ref{uistklein}) holds.
%
%
%
%
%
%
%
We have
\begin{align}\label{5.11}
\frac{d}{dt}E(t)=
\frac{d}{dt}E_1(t)+\frac{d}{dt}E_2(t).
\end{align}
First we consider
$\frac{d}{dt}E_1(t) = I_1 + I_2 + I_3$.
Define $\mu_2 = \frac{1}{L}$.
By the definition of $h_2$
in
(\ref{h2defi})
we have
$(h_2)_x = - \mu_2 \, h_2$ and thus
\begin{equation}
\label{5.18}
I_1  = -\int_0^L \left( a^2-(\bar u+u)^2\right) \, \mu_2 \, h_2 \, u_x^2 + \mu_2 \, h_2 \, u_t^2 \, dx.
\end{equation}
%
For all $x\in [0,\, L]$ we have
$\mu_2 \, h_2(x) \geq \frac{1}{{\rm e} \, L}$
hence we have
%
%
\[I_1 \leq -  \frac{1}{2\,{\rm e} \, L\,  k }\, \int_0^L  2\,  h_1  \, (a^2-(\bar u+u)^2) \, u_x^2 + 2\,  h_1 \, u_t^2   \, dx.\]

Thus, by (\ref{5.38a}) we have
\begin{align}\label{5.51}
I_1\leq -  \frac{1}{2\,{\rm e} \, L\,  k }\, E_1(t).
\end{align}

Now we consider the term $I_2$ as defined in (\ref{5.16}).
Note that  due to
(\ref{2.27}), each of the terms that are added in $I_2$,
in particular $F\, u_t$ and $F\, u_x$,
contains a second order term of $u$, $u_t$, $u_x$ as a factor,
that is $u \,u_t$, $u \, u_x$, $u_x \, u_t$,  $u_x^2$  or $u_t^2$.

More precisely, the terms that appear as factors are either third order
terms $u_t u_x^2$, $u_x u_t^2$, $u_x^3$, $u u_x^2$, $u u_x u_t$, $u
u_t^2$ or terms of the form $\theta\bar u \, u u_t$, $\theta\bar
u \, u u_x$ , $\theta\bar u \, u_x u_t$, $\theta\bar u \,
u_x^2$ or $\theta\bar u \, u_t^2$. Since we have $h_1(x)= k$
and $\max_{x\in[0,L]} |h_2(x)|= 1$, the definition of
$I_2$ implies that there exists a continuous function $P_0$ with
$P_0(0)=0$ such that we have an estimate of the form
\[
I_2
\leq
 P_0\left( \max_{x\in [0,L]} \left\{|u(t,x)|, |u_x(t,x)|,  |u_t(t,x)|,\, |\bar u(x)|\right\} \right)
 \int_0^L \left(u^2+u_t^2+u_x^2\right) \,dx.
\]

In fact, the definitions of $I_2$, $F$ and $\tilde F$ imply that we can choose
\begin{eqnarray}\label{P_0}
P_0(t) & = &
 4\,  k \, t^2 + 2 k\Big(1+\frac{\theta}{2}\Big) t + 4 \, k\Big(1+\frac{\theta}{2}\Big) t^2 + 2 t + 4\Big(1+\frac{\theta}2\Big) t^2
 \nonumber\\
 & + & 2(k + 1)
\Big[ \theta^2 \frac{ 3 a^4 t^4 + 2 a^2 t^6 + t^8} {2(a^2 - t^2)^3}
+ \theta \frac{a^2 t}{ a^2 - t^2} + \theta \frac{t^4 + 3 a^2 t^2}{2
(a^2 - t^2)} \nonumber\\ & + &
 \theta^2  \frac{2 t^7 + 3 a^2 t^5 + 3 a^4 t^3}{4(a^2 - t^2)^3} t
+ 2 t^2 + \theta \frac{3 a^2 t + t^3}{a^2 - t^2} t \nonumber\\ & + &
2t + 2 t^2 + \frac{3}{2} \theta t^2 + \theta t \Big].
\end{eqnarray}

Using (\ref{poin}), and then (\ref{5.6}),  (\ref{5.7})  we obtain
the inequality
\begin{eqnarray}
\nonumber
I_2 & \leq &
 P_0\left( T_{Li}(t) \right)
 (1 + L^2)
 \int_0^L \left(u_t^2+u_x^2\right) \,dx
\\
\label{5.19}
& \leq &
 P_0\left( T_{Li}(t)
 \right)
(1 + L^2)
\left(\frac{1}{K_1} + \frac{1}{\tilde K_1} \right) E_1(t).
\end{eqnarray}

Now we focus on the boundary term $I_3$.
We use the notation $\bar u_0:=\bar u(0)$ and $\bar u_L:=\bar u(L)$.
Since $k>0$, by the boundary
conditions (\ref{2.29}), (\ref{2.30}) we have
\begin{eqnarray}\label{5.20}
\label{5.20imJahr2016}
I_3 & = & I_3^L -I_3^0
\end{eqnarray}
where
\begin{eqnarray}
\label{i3ldefinition}
I_3^L & = &
 -(a^2-\bar u_L^2)\, e^{-1} u_x^2(t,L)
 \\
\label{i30definition}
I_3^0 & = &  \Big( a^2 - \left(  (\bar u_0 + u(t,0) )+ \frac{1}{k} \right)^2 \Big) \, u_x^2(t,0).
\end{eqnarray}
Since (\ref{i3vor}) holds, we have
$|\bar u_0 + u(t,0)| \leq |\bar u_0|  + \gamma \, a - |\bar u_0|
 = \gamma a \leq a - \frac{1}{k}$.
%
%
Hence we have $I_3^0\geq 0$.
Since $I_3^L\leq 0$,  due to (\ref{5.20imJahr2016}) this implies
\begin{align}\label{5.22}
I_3 \leq 0.
\end{align}
Then inequalities (\ref{5.51}), (\ref{5.19}) and (\ref{5.22}) yield
\begin{eqnarray}\label{5.41}
\frac{d}{dt} \, E_1(t) &\leq &
 -  \frac{1}{2\,{\rm e} \, L\,  k } \, E_1(t)+
P_0\left( T_{Li}(t) \right)
 (1 + L^2)
\left(\frac{1}{K_1} + \frac{1}{\tilde K_1} \right) E_1(t) \nonumber\\
& = & - \left( \frac{1}{2\,{\rm e} \, L\,  k }   - P_0\left( T_{Li}(t) \right)
 (1 + L^2)
\left(\frac{1}{K_1} + \frac{1}{\tilde K_1} \right) \right) E_1(t).
\end{eqnarray}
With (\ref{kappadefinition}) this implies
that $E_1(t)$ is a strict Lyapunov function
and
(\ref{e1lyapunov}) holds.

Similarly, for $\tilde{I}_1$, we infer
\begin{eqnarray}
\label{5.27}
\tilde{I}_1 & = &
- \int_0^L (a^2-(\bar u+u)^2)
\, \mu_2 \, h_2 \, u_{xx}^2  +\mu_2 \, h_2  \, u_{tx}^2\, dx \nonumber
\\
& \leq &
 - \frac{1}{2\,{\rm e} \, L\,  k }\,
 \int_0^L (a^2-(\bar u+u)^2)
\, 2\, h_1 \, u_{xx}^2  + 2  \, h_1  \, u_{tx}^2\, dx. \nonumber
\end{eqnarray}
Hence (\ref{5.38agugat})
yields
\begin{eqnarray}
\tilde{I}_1
&\leq  &
-  \frac{1}{2\,{\rm e} \, L\,  k }\,  E_2(t).
\end{eqnarray}
Now we consider $\tilde I_2$ as defined in
(\ref{5.25}).
All the terms that are added in
$\tilde I_2$ are
contain  factors
$u_{xx}\,u_{tx}$, $u_{xx}^2$, $u_{tx}^2$,
$F_x \, u_{xx}$,  $F_x \, u_{tx}$.
Except for $F_x \, u_{xx}$,  $F_x \, u_{tx}$, it can easily be seen
that the coefficients that are multiplied
with these quadratic terms
become arbitrarily small if if  $T_{Li}(t)$
as defined in (\ref{T(t)}) is sufficiently small.

%
%
%
Now we have a closer look at $F_x$.
From
(\ref{2.27}),
we have
\begin{eqnarray*}
F_x&=&-2u_xu_{tx}-2u_tu_{xx}-2u_x^3-4uu_xu_{xx}-3\theta
uu_{x}^2-\frac{3}{2}\theta u^2u_{xx}-\theta u_xu_t-\theta uu_{tx}\\
&-&\theta^3\frac{9a^6\bar u^3-2\bar u^9-6a^4\bar u^5-11a^2\bar
u^7}{8(a^2-\bar u^2)^4}u^2-\theta^2\frac{\bar u^7+6a^4\bar
u^3-3a^2\bar u^5}{2(a^2-\bar u^2)^3}uu_x\\
&-&\theta\frac{\bar u^3}{a^2-\bar u^2}u_x^2-4\bar u u_x
u_{xx}-\theta\frac{3a^2\bar u-\bar u^3}{a^2-\bar u^2}(u_x^2+u
u_{xx})\\
&-&\theta^3\frac{6a^6\bar u^6+4a^2\bar u^{10}-\bar u^{12}-3a^4\bar
u^8}{2(a^2-\bar u^2)^5}u-\theta^2\frac{a^4\bar u^3+a^2\bar
u^5}{2(a^2-\bar u^2)^3}u_t-\theta^2\frac{3a^4\bar u^4}{(a^2-\bar
u^2)^3}u_x\\
&-&\theta\frac{a^2\bar u}{a^2-\bar u^2}u_{tx}-\theta\frac{3a^2\bar
u^2+\bar u^4}{2(a^2-\bar u^2)}u_{xx}.
\end{eqnarray*}
Also in $F_x \, u_{xx}$,  $F_x \, u_{tx}$
all the  terms that are added contain quadratic factors
$u_{xx}\,u_{tx}$, $u_{xx}^2$, $u_{tx}^2$,
$u_x\,u_{xx}$, $u_x\,u_{tx}$, $u_t\,u_{xx}$, $u_t\,u_{tx}$,
$u\,u_{xx}$, $u\,u_{tx}$
and the coefficients that are multiplied
with these factors
become arbitrarily small if if  $T_{Li}(t)$
as defined in (\ref{T(t)}) is sufficiently small.
%
%
%
%
Thus similar as in the estimate of $I_2$,
we can  find  a
continuous function $P_1(t)$ with $P_1(0)=0$
such that
using (\ref{poin}), and then (\ref{5.6}), (\ref{5.7})  we obtain the
inequality
\begin{eqnarray}
\label{5.28} \tilde I_2 & \leq &
 P_1\left( T_{Li}(t) \right)
 (1 + L^2)\int_0^L \left(u_t^2+u_x^2+u_{tx}^2+u_{xx}^2\right) \,dx
\nonumber\\ & \leq &
 P_1\left( T_{Li}(t)
 \right)
(1 + L^2)\left(\frac{1}{K_1} + \frac{1}{\tilde K_1} \right)
(E_1(t)+E_2(t)).
\end{eqnarray}
In fact if we replace in the representation of
$F_x$ in each of the terms that are added except for one factor
the expressions
$u$, $u_x$, $u_t$, $u_{xx}$, $u_{tx}$, $\bar u$
by $t$,
and treat
the other terms from the definition of $\tilde I_2$
in a similar way
since $h_1=k$ and $|h_2| \leq 1$
we can choose
\begin{eqnarray}\label{P_1}
P_1(t) & = &
  8\, t + 2 \, t  + 8 \, t^2 +  4 \, k \, t^2  + 4 \, k \, t
 \nonumber\\
 & + & 2(k + 1)
\Big[\theta^3\frac{9a^6t^2+2t^9+6a^4t^5+11a^2t^7}{8(a^2-t^2)^4}t+\theta^2\frac{t^7+6a^4t^3+3
a^7}{2(a^2- t^2)^3}t\nonumber\\
& +&
\theta\frac{t^3}{a^2-t^2}t+4t^2+\theta\frac{3a^2t+t^3}{a^2-t^2}2t+\theta^3\frac{6a^6t^6+4a^2t^{10}+t^{12}+3a^4t^8}{2(a^2-t^2)^5}\nonumber\\
& + &
\theta^2\frac{a^4t^3+a^2t^5}{2(a^2-t^2)^3}+\theta^2\frac{3a^4t^4}{(a^2-t^2)^3}+\theta\frac{a^2t}{a^2-t^2}+\theta\frac{3a^2t^2+t^4}{2(a^2-t^2)}
\nonumber\\ & + & 2t + 2 t + 2 t^2+4 t^2+ 3\theta t^2+\frac{3}{2}
\theta t^2 + \theta t +\theta t\Big].
\end{eqnarray}

For the boundary term $\tilde{I}_3$, we
use (\ref{2.26}) in the form
\[
(a^2 - (\bar u + u)^2) \, u_{xx} =  u_{tt} + 2 (\bar u + u) u_{tx} -  F.
\]
In particular, for $x=L$
due to (\ref{2.30})  we have $u(t,L) = u_t(t,L)= u_{tt}(t,L) = 0$,
hence  at $x=L$ we get
\[(a^2 - \bar u_L^2) \, u_{xx}(t,L) =  0 +  2  \,\bar u_L \, u_{tx}(t,L) -  F(t,L).\]
Using (\ref{2.27}) and the definition of $\tilde F$ we  obtain
 \[F(t,L) =
- \bar u_L \,\left( \frac{3}{2} \, \theta \, \bar u_L
 + 4 \, \bar u_x(L) \right) u_x(t,L) - 2 \, \bar u_L \,  u_x^2(t,L).
 \]
Due to (\ref{i3vor}) this yields
\begin{equation}
\label{ftlschranke}
|F(t,L)| \leq
\bar u_L \, \left(\frac{3}{2} \, \theta \, \bar u_L  + 4 |\bar u_x(L)| + 2\right) |u_x(t,L)|.
\end{equation}

We have
\[\tilde{I}_3= \tilde{I}_3^L - \tilde{I}_3^0\]
where $\tilde{I}_3^0$
is given in (\ref{tildei30definition}) and
\begin{eqnarray*}
 \tilde{I}_3^L & = &
 \left[
 2 \,k \,  u_{tx}(t,\, L)  -  \frac{1}{\rm e}\, u_{xx}(t,\, L)
 \right]
 \left[ 2  \,\bar u_L \, u_{tx}(t,\, L) -  F(t,\, L) \right]
 - \left( 2 \,  k \, \bar u_L + \frac{1}{\rm e} \right) u_{tx}^2(t,\, L)
   \\
   & = &
   \left[
 2 \,k \,  u_{tx}(t,\, L)  -  \frac{1}{\rm e}\,
 \frac{ 2  \,\bar u_L \, u_{tx}(t,\, L) -  F(t,\, L) }{a^2 - \bar u_L^2}
 \right]
 \left[ 2  \,\bar u_L \, u_{tx}(t,\, L) -  F(t,\, L) \right]
 \\
 &
 -
 &
  \left( 2 \,  k \, \bar u_L + \frac{1}{\rm e} \right) u_{tx}^2(t,\, L)
 \\
    & = &
   \left[    2\, k \,\bar u_L  -  \frac{1}{\rm e}\, \frac{a^2 + 3 \bar u_L^2 }{a^2 - \bar u_L^2}    \right]  u_{tx}^2(t,\, L)
   \\ & + &
 \left[  \frac{1}{\rm e}\, \frac{4 \bar u_L }{a^2 - \bar u_L^2} - 2 \, k \right] u_{tx}(t,\, L)\,  F(t,\, L)
 - \frac{1}{\rm e}\, \frac{ 1 }{a^2 - \bar u_L^2} \,  F(t,\, L)^2.
 \end{eqnarray*}
 With $C_g(\bar u_L)$ as defined in (\ref{cgdefinition}) we have
\begin{equation}
\label{cgdefinitionneu}
C_g(\bar u_L) =
 \frac{ a^2 - \bar u_L^2}{{\rm e}\, \bar u_L^2 }
 \,\frac{1}{\left(
 \frac{3}{2} \, \theta \, \bar u_L + 4 \, |\bar u_x(L)| + 2 \right)^2}
=
 \frac{ a^2 - \bar u_L^2}{{\rm e}\, \bar u_L^2\left( 2
+ \frac{3}{2} \, \theta \, \bar u_L +  \frac{ 2\,\theta\,\bar u_L^3 }{  a^2 - \bar u_L^2 } \right)^2}
>0.
\end{equation}

 Then due to
(\ref{ftlschranke}) we have
$C_g(\bar u_L) \, F(t,L)^2 \leq  (a^2 - \bar u_L^2) \,{\rm e}^{-1} \,  u_x^2(t,L)$.
Hence
(\ref{i3ldefinition}) implies
\begin{equation}
\label{i3cgf2ungleichung}
I_3^L + C_g(\bar u_L) \, F(t,L)^2 \leq
\left[ - \frac{1}{\rm e} (a^2 - \bar u_L^2) +
\frac{1}{\rm e} (a^2 - \bar u_L^2)
\right] \, u_x^2(t,L)
=
 0.
 \end{equation}

Consider the matrix $\tilde A_3(\bar u_L)$
 as defined in (\ref{tildea3definition}).
 With the notation
$F= F(t,L)$ and $u_{tx}=u_{tx}(t,\, L)$
 we have
%
 \[
  \tilde{I}_3^L -C_g \, F^2
  =
  -( u_{tx}, \, F) \, \tilde A_3(\bar u_L) \,\left(
 \begin{array}{c}
 u_{tx}
 \\
 F
 \end{array}
 \right).
 \]

  Due to (\ref{uistklein}), Lemma \ref{neulemmaneueversion} implies that
  the matrix $\tilde A_3(\bar u_L)$ is positive definite.
 Thus $\tilde{I}_3^L - C_g \, F^2 \leq 0$. Due to (\ref{i3cgf2ungleichung})
 this yields
 \begin{equation}
 \label{tildei3lungleichung}
 I_3^L + \tilde{I}_3^L  = I_3^L + C_g \, F^2(t,L)  + \tilde{I}_3^L - C_g \, F^2(t,L) \leq 0.
 \end{equation}


Now we look at $\tilde{I}_3^0$ that depends on the values at $x=0$
where the Neumann condition (\ref{2.29}) is prescribed
and we have
$u_t(t,0) = \frac{1}{k}\, u_{x}(t,0)$,
$ u_{tt}(t,0) = \frac{1}{k}\, u_{xt}(t,0)$.
Hence
with the notation
$F(t,\,0)=F(u(t,\, 0)\, , u_x(t,\, 0),\, u_t(t,\, 0)\,)$ and $\bar u(t,\, 0)= \bar u_0$
for $x=0$,  (\ref{2.26}) yields
\begin{equation}
\label{hilfsgleichung15072016}
u_{xx}(t,\,0)=
\frac{ 1 +  2\, k \,  (\bar u_0 + u(t,\,0)) }{k\,(a^2 - (\bar u_0 + u(t,\,0))^2)}
\, u_{tx}(t,\,0)
-
\frac{ 1  }{(a^2 - (\bar u_0 + u(t,\,0))^2)} \, F(t,\,0).
\end{equation}
Due to (\ref{2.27}) we have for $x=0$
\begin{eqnarray*}
F(t,0) & = &
-2 \left[ \frac{1}{k} + \bar u_0 + u \right]\, u_x^2
\\
& - &
\left[\theta \, \frac{\bar u_0^4+3a^2\bar u_0^2 + \frac{2}{k} a^2 \bar u_0 }{2(a^2-\bar u_0^2)}
+
\frac{\theta}{k}\, u + \frac{3}{2}\, \theta \, u^2 +
\theta \, \frac{3a^2\bar u_0-\bar u_0^3}{a^2-\bar u_0^2}\, u
\right]\,u_x
\\
& -  &
 \theta^2 \frac{3a^4\bar u_0^4-2a^2\bar u_0^6+\bar u_0^8}{2(a^2-\bar u_0^2)^3} \,u
  -
 \theta^2 \, \frac{2\bar u_0^7-3a^2\bar u_0^5+3a^4\bar u_0^3}{4(a^2-\bar u_0^2)^3}\,u^2.
\end{eqnarray*}
Due to (\ref{i3vor}) this yields
\begin{eqnarray*}
\label{ft0schranke}
|F(t,0)|
& \leq &
2 \left[ \frac{2}{k^2} +  \frac{\bar u_0}{k} \right]\, \left|u_x(t,\, 0)\right|
\\
& + &
\left[\theta \, \frac{\bar u_0^4+3a^2\bar u_0^2 + \frac{2}{k} a^2 \bar u_0 }{2(a^2-\bar u_0^2)}
+
 \frac{5}{2}\,\frac{\theta}{k^2}   +
\frac{\theta}{k} \, \frac{3a^2\bar u_0-\bar u_0^3}{a^2-\bar u_0^2}
\right]\,\left| u_x(t\,,0)\right|
\\
& + &
 \frac{\theta^2 }{k}
 \left[\frac{6 k a^4\bar u_0^4- 4 k a^2\bar u_0^6+ 2 k \bar u_0^8
  +
  2\bar u_0^7-3a^2\bar u_0^5+3a^4\bar u_0^3}{4(a^2-\bar u_0^2)^3}\right]\, \left| u(t,\, 0)\right|.
\end{eqnarray*}
With $K_{\partial}(k,  \,\bar u_0 )$ as defined in
(\ref{upsilondefinition})
 due to Young's inequality we have
\begin{equation}
\label{ft0inequality}
F(t,0)^2 \leq K_{\partial}(k,  \,\bar u_0 ) \, u_x^2(t,\,0) + C_{E1}(\bar u_0)   \, u(t,\,0)^2
\end{equation}
with the constant
\begin{equation}
\label{ce1definition}
C_{E1}(\bar u_0)
=
2 \,
\frac{\theta^4 }{k^2} \,
 \left[\frac{6 k a^4\bar u_0^4- 4 k a^2\bar u_0^6+ 2 k \bar u_0^8
  +
  2\bar u_0^7-3a^2\bar u_0^5+3a^4\bar u_0^3}{4(a^2-\bar u_0^2)^3}\right]^2.
 \end{equation}

%
With the notation $h_{\tilde B_3}(t) = \bar u_0 + u(t,0)$
using (\ref{hilfsgleichung15072016}) we obtain
\begin{eqnarray}
\nonumber
\tilde{I}_3^0 & = &
-\frac{1}{a^2- h_{\tilde B_3}(t)^2} \, F^2(t,0)
+
2 \,
\left(\frac{1+2k\, h_{\tilde B_3}(t)}{k\, (a^2-h_{\tilde B_3}(t)^2)}- k\right)\,F(t,0)\, u_{tx}(t,0)\\
\label{tildei30definition}
& + &
\left(
1+ 2\, k\, h_{\tilde B_3}(t) -
\frac{(1+2\, k \, h_{\tilde B_3}(t))^2}{k^2\, (a^2- h_{\tilde B_3}(t) ^2)}
\right) \, u_{tx}^2(t,0).
 \end{eqnarray}

%
%
%
For $ \upsilon  := 2 \, k^2> k^2$ we have
\[
  \tilde{I}_3^0  + \upsilon \, F(t,0)^2
  =
  ( u_{tx}(t,\,0), \, F(t,\,0)) \, \tilde B_3(h_{\tilde B_3}(t))  \,\left(
 \begin{array}{c}
 u_{tx}(t,\,0)
 \\
 F(t,\,0)
 \end{array}
 \right)
 \]
with the matrix $\tilde B_3(h_{\tilde B_3}(t))$ from (\ref{tildeb3definition}).
%
%
Since we have assumed that  $|\bar u_0|< \varepsilon_1(2\, k^2)$ and
$|u(t, \,0) |<   \varepsilon_1(2\,k^2)$,
by the definition of $h_{\tilde B_3}(t)$ this implies
$|h_{\tilde B_3}(t)| < 2\,\varepsilon_1(2\, k^2)$
hence Lemma \ref{neulemmaneueversion}
implies that the matrix $\tilde B_3 ( h_{\tilde B_3}(t))$ is positive definite.
Thus we have $\tilde{I}_3^0  + \upsilon \, F(t,0)^2\geq 0$.
Hence
due to (\ref{i30definition}),
(\ref{i3vor}),
(\ref{ft0inequality}) and
 (\ref{kpartialassumption}) we have
 \begin{eqnarray*}
 {I}_3^0 + \tilde{I}_3^0
& = & {I}_3^0 - 2 \, k^2 \, F(t,0)+ \tilde{I}_3^0  + 2\, k^2 \, F(t,0)
\\
& \geq & {I}_3^0 - 2 \, k^2 \, F(t,0)
\\
& \geq &
 \Big( a^2 - \left(  \bar u_0 + \frac{2}{k} \right)^2 
- 2\,  k^2 \,K_{\partial}(k,  \,\bar u_0 ) \Big) \, u_x^2(t,\,0) - 2 \, k^2 \, C_{E1}(\bar u_0)   \, u(t,\,0)^2
\\
& \geq & - 2 \, k^2 \, C_{E1}(\bar u_0)   \, u(t,\,0)^2.
\end{eqnarray*}
Due to the Dirichlet boundary condition (\ref{2.30}) we have
\[u(t,\,0)^2 = \left|u(t,\,0) - u(t,\, L) \right|^2
=
 \left|\int_0^L u_x(t,\, s)\, ds\right|^2 \leq L \, \int_0^L u_x^2 (s) \,ds
\leq \frac{  L}{K_1} E_1(t)
\]
 where the last inequality follows from (\ref{5.6}).
This yields
\begin{equation}
\label{i30tildeschranke}
 {I}_3^0 + \tilde{I}_3^0  \geq -  2 \, k^2 \, C_{E1}(\bar u_0) \, \frac{  L}{K_1} E_1(t)
 .
 \end{equation}
We have
\[
 \frac{d}{dt} E(t)
=
I_1 + I_2 + {\tilde I}_1 + {\tilde I}_2 +
(I_3^L  + {\tilde I}_3^L)
-
(I_3^0 +  {\tilde I}_3^0).
 \]
Then inequalities
(\ref{5.51}),  (\ref{5.19}),
(\ref{5.27}),
(\ref{5.28}), (\ref{tildei3lungleichung}) and (\ref{i30tildeschranke}) yield
\begin{eqnarray}
\nonumber
\frac{d}{dt} E(t)
\nonumber
& \leq &
- \left[ \frac{1}{2\,{\rm e} \, L\,  k }   - P_0\left( T_{Li}(t) \right)
 (1 + L^2)
\left(\frac{1}{K_1} + \frac{1}{\tilde K_1} \right) \right] \,  E_1(t)
\\
\label{5.46}
& - &   \left[  \frac{1}{2\,{\rm e} \, L\,  k }
 - P_1\left( T_{Li}(t) \right)
 (1 + L^2)
\left(\frac{1}{K_1} + \frac{1}{\tilde K_1} \right) \right] \, E_2(t)
\\
\nonumber
& + &
\left[P_1\left( T_{Li}(t) \right)
 (1 + L^2)\left( \frac{1}{K_1} + \frac{1}{\tilde K_1} \right)
   + 2 \, k^2 \, C_{E1}(\bar u_0) \, \frac{  L}{K_1}  \right] \, E_1(t).
\end{eqnarray}


Note that due to (\ref{apriori}),
$T_{Li}(t)$
becomes arbitrarily small if the norm of the initial data and of $\bar u$ is sufficiently small.
%
%
%
%
%
Define the number
$\kappa>0$ as in (\ref{kappadefinition}) and
$\mu$ as in (\ref{mudefinition}).
%
%
%
%
%
%
%
%
%
%
%
%
%
%
%
%
%
Then
(\ref{5.46})
 and  the definition of $\mu$
 yield
\begin{align}\label{5.32}
\frac{d}{dt}E(t)\leq -{\mu} \,E(t)
\;\mbox{\rm for all}\; t\in[0,T],
\end{align}
which  implies inequality (\ref{4.9}).
\end{proof}


\subsection{Proof of Corollary \ref{cor4.3}}

Now we present the proof of Corollary  \ref{cor4.3}.
\begin{proof}
For all $t\in [0,\,T]$
the inequalities (\ref{unter}) and (\ref{poin}) for the
$H^2$-Lyapunov function $E(t)$ defined in (\ref{4.1})
imply the following inequalities
\begin{align}\label{5.33}
&\|u(t,\cdot)\|_{H^2(0,L)}\leq  \sqrt{ \frac{1 + 2 L^2}{K_{\min} } }   \sqrt{E(t)},\\
\label{5.40}
& \|u_t(t,\cdot)\|_{H^1(0,L)}\leq  \sqrt{ \frac{1}{K_{\min} } }  \sqrt{E(t)}.
\end{align}
Inequality (\ref{ober}) implies for $t=0$:
\begin{align}\label{5.34}
\sqrt{ K_{\max}   } \; \|(u(0,\cdot),u_t(0,\cdot))\|_{H^2(0,L)\times
H^1(0,L)}\geq\sqrt{E(0)}.
\end{align}
With the  positive constants
\begin{align}
\tau_1:=\frac{K_{\min}}{1+2L^2},\quad
\tau_2:=K_{\max}\quad,  
\label{5.35}
\eta_1 := 2\, \sqrt{\tau_2/\tau_1},
\end{align}
the inequalities (\ref{5.33}), (\ref{5.40}),  (\ref{4.9}), (\ref{5.34}) imply the
estimate (\ref{4.10}).

The inequality (\ref{4.11}) follows from (\ref{4.10}) and the
Sobolev embedding $H^2((0, L))\hookrightarrow C^1([0, L])$ and
$H^1((0, L))\hookrightarrow C^0([0, L])$, see
(\cite{ABM},\cite{TW}).

Theorem \ref{thm4.2} implies that
for all $T>0$ the decay rate
$\mu = \frac{1}{4 \,{\rm e} L k}$
can be achieved for sufficiently small initial data.
With this decay rate, for
\begin{equation}
\label{Tschranke0305}
T\geq \max\limits_{i\in \{1,\, 2\}} 8 \,{\rm e} \,L k {\,{\rm ln}( 2\, \eta_i)},
\end{equation}
and $i\in \{1,\, 2\}$ we have the inequality
$\eta_i \, \exp\left(-\frac{\mu}2 \,T\right) \leq \frac{1}{2}$.
Hence
the inequalities (\ref{4.10}) and   (\ref{4.11})
imply
the inequalities (\ref{4.1002.05a}) and
(\ref{4.1002.05b}).
%
\end{proof}
%
%
%
%

%

\section{Global solutions}
\label{sec6}
In this section we show that
the exponential  decay of the  Lyapunov function
defined in  (\ref{4.1}) implies that
the solution exists global in time without
losing regularity, that is it keeps the regularity
of the initial state.

Let us first observe that (\ref{2.5}) implies that the eigenvalues
$\lambda_-= (-a+ \bar u + u)$ and $\lambda_+=(a+ \bar u + u)$ do not depend on the derivatives of $u$.
Therefore for a given value of $T>0$,
$(s,x,t)\in [0,T]\times [0,L]\times [0,T]$,
 the field of characteristic curves $\xi^u_{\pm}(s,x,t)$
 corresponding to $u\in C^1([0,T]\times [0,L])$
defined by the integral equation
\begin{equation}
\xi^u_{\pm}(s,x,t)=x             \pm a(s-t) +\int_t^s \bar u(\xi^u_{\pm}(\tau,x,t))
+u (\tau,\,\xi^u_{\pm}(\tau,x,t))\,d\tau
\in [0,L]
\end{equation}
is well-defined
for a $C^1$-function $u$
if $\bar u$ and $u$ have sufficiently small $C^1$-norm.

%
%
%


%
%
In order to obtain a semi-global $C^1$-solution of
 (\ref{2.26}), (\ref{2.29}), (\ref{2.30})
 in the sense of integral equations along these characteristic curves,
 the boundary condition
 (\ref{2.29}) at $x=0$ is written in
the form of the integral equation
\begin{equation}
\label{rbx=0h}
u(t,0)= u(0,0) + \frac{1}{k} \int_0^t u_x(s,0)\, ds.
\end{equation}
To be precise we define  $(r_+,r_-)=(R_+-\bar R_+, R_- - \bar R_-)$.
Then  we have $u = - \frac{1}{2}(r_+ + r_-)$.
Thus the boundary condition
(\ref{2.30}) at $x=L$ is equivalent to
\begin{equation}
\label{rand1}
r_-(t,L) = - r_+(t,L)
\end{equation}
 and
(\ref{rbx=0h}) is equivalent to
\begin{equation}
\label{rand2}
r_+(t,0) = -r_-(t,0) +
(r_+(0,0) + r_-(0,0) ) +  \frac{1}{k} \int_0^t (r_+)_x(s,0) + (r_-)_x(s,0)\, ds.
\end{equation}
Due to (\ref{2.6}), $(r_+,r_-)$ satisfies the system in diagonal form
\begin{equation}
\label{117}
\partial_t\left(\begin{array}{ll}
r_+\\r_-\end{array}\right)+\hat{D}(\bar R_+ + r_+, \bar R_- + r_-)\
\partial_x\left(\begin{array}{ll}
r_+\\r_-\end{array}\right)
\end{equation}
\begin{align}
=\hat{S}(\bar R_+ + r_+, \bar R_- + r_-)
- \hat{S}(\bar R_+, \bar R_-)
\end{align}
\begin{equation}
\label{119}
+
\left[\hat{D}(\bar R_+ , \bar R_-) -
\hat{D}(\bar R_+ + r_+, \bar R_- + r_-)
\right]
\partial_x
\left(\begin{array}{ll}
\bar R_+\\ \bar R_-
\end{array}\right).
\end{equation}
Let $t^u_{\pm}(x,t)\leq t$ denote the time
where $\xi^u_{\pm}(s,x,t)$
hits the boundary of $[0,T]\times [0,L]$.
Then
(\ref{shutdefinition}) implies that
 $(r_+,r_-)$ satisfy the integral equations
\begin{eqnarray*}
r_\pm(t,x)
& = &
r_\pm(t^u_\pm(x,t),\; \xi^u_\pm(t^u_\pm(x,t),\;x,t))
\\
& + &
\int_{t^u_\pm(x,t) }^{t} p_\pm(r_+ + r_-) ( s,\, \xi^u_{\pm}(s,\;x,t))\, ds
\end{eqnarray*}
with
\begin{equation}
p_\pm(z)
=
\frac{ \theta}{2}
[\frac{1}{4} z^2  - \bar u \, z ]
+ \frac{1}{2} [ \partial_x\bar R_\pm ]
z .
\end{equation}

Now we consider the initial boundary value problem
with initial data for $(r_+,r_-)$ at $t=0$,
the equation in diagonal form
(\ref{117})-(\ref{119})
and the boundary conditions
(\ref{rand1}), (\ref{rand2}).

Let a time $T>0$ be given such
that (\ref{4.1002.05a})  holds.
With initial data for $(r_+,r_-)$ in $[C^1([0,L])]^2$ at
$t=0$ that are sufficiently small
(with respect to the $C^1$-norm),
compatible to $u$
and satisfy the
 $C^1$--compatibility conditions for (\ref{rand1}), (\ref{rand2}),
 as in \cite{LY}
we obtain a semi-global classical solution $(r_+,r_-)\in C^1([0,T]\times [0,L])$.
Thus we also get  a
continuously differentiable function  $u = - \frac{1}{2}(r_+ + r_-)$
for $(t,x)\in [0,\, T]\times [0,\, L]$.


If  our initial data for $(r_+,r_-)$ at $t=0$ are more regular,
namely in  $[ H^2(0,L) ]^2$
and satisfy the assumptions that we just mentioned,
they  generate a solution that is more regular
than the classical solution in general:
For all $t\in [0,T]$ the second partial derivatives
$(\partial_{xx} r_+,\, \partial_{xx} r_-)$
are in $L^2(0,L)$.
This can be seen as follows.

For initial
data with sufficiently small $H^2$-norm
also the $C^1$-norm is small.
Thus we  know that a classical semi-global solution
exists on $[0,T]$ and
we can fix the corresponding characteristic curves.
As a consequence, we obtain a semilinear evolution  for
$(\partial_{xx} r_+,\, \partial_{xx} r_-)$
with fixed characteristic curves.
The evolution of
$\partial_{xx} r_{\pm}$
is governed by the integral equation
\begin{eqnarray*}
\label{rxx1}
\partial_{xx}  r_\pm(t,x)
& = &
\partial_{xx} r_\pm(t^u_\pm(x,t),\; \xi^u_\pm(t^u_\pm(x,t),\;x,t))
\\
\label{rxx2}
& + &
\int_{t^u_\pm(x,t) }^{t} P_\pm(r_+, r_-, \partial_x r_+, \partial_x r_-, \partial_{xx} r_+, \partial_{xx} r_-) (s,\, \xi^u_{\pm}(s,\;x,t))\, ds
\end{eqnarray*}
with a polynomial $P_\pm$
(with $C^1$-coefficients, similar to $p_\pm$)
that is affine linear with respect to $\partial_{xx} r_{\pm}$.
We can consider
$t^u_\pm(x,t)$,  $\xi^u_\pm(s,x,t)$,
 $r_+, r_-, \partial_x r_+, \partial_x r_-$
 as given continuous functions. Then
$\partial_{xx}  r_\pm(t,\cdot)\in L^2(0,L)$
 is given as the solution of a family of linear integral equations.
 Thus we can show that
 $\partial_{xx}  r_\pm(t,\cdot)\in L^2(0,L)$ for all $t\in [0,T]$.
This implies that the lifespan  of the
$H^2$-solution
only depends on the $C^1$-norm of $(r_+,r_-)$,
so in particular it is well-defined on the time interval $[0,T]$.

Thus we can construct a global $H^2$-solution as follows:

Since we have  chosen $T>0$ such
that
(\ref{4.1002.05a})
 holds,
for nonzero initial data
$(\varphi,\, \psi)$ with sufficiently small $H^2\times H^1$-norm,
we obtain an $H^2$-solution on $[0,T]$ as described above
%
and due to
(\ref{4.1002.05a})
the $H^2\times H^1$ norm
of  $(u(T,\cdot), \, u_t(T,\,\cdot))$
is less than
the
$H^2\times H^1$ norm
of  $(u(0,\cdot), \, u_t(0,\,\cdot))$.
Thus we can start our construction again
 with initial data
 $(u(T,\cdot), \, u_t(T,\,\cdot))$
 where the $H^2\times H^1$ norm has been decreased
 at least  by a factor $\frac{1}{2}$
  to obtain an $H^2$-solution
on the time interval $[T,\, 2\,T]$.
By repeating the procedure iteratively,
for initial data
$(\varphi,\,\psi)$
 with sufficiently small
$H^2\times H^1$-norm,
we thus obtain a solution that is
well defined for all $t>0$.
Moreover, Lemma 2 in  \cite{gutus} implies that
the
$H^2\times H^1$-norm of the solution decays exponentially with time.
In addition,
Theorem \ref{thm4.2} implies that
also the Lyapunov function $E(t)$ decays exponentially with time.
The above considerations yield the following result.
\begin{thm}
\label{thm4.2h}
{\bf (Global Exponential Decay of the $H^2$-Lyapunov Function).}
\label{thm1h}
\noindent
Let a stationary subsonic state $\bar u(x)\in C^2(0,L)$ be given
that satisfies (\ref{statvor}). 
Let $\gamma \in (0, 1/2]$ be given.
Assume that for all $x\in L$ we have
   $\bar u(x) \in\left(0, \,  \gamma \, a\right)$.
Choose a real number $k> \frac{1}{ (1 - \gamma) \, a  }$.
Assume that $\bar u$ is sufficiently small and
$k$ sufficiently large such
that
for $K_{\partial}(k,  \,\bar u(0) )$  as defined in (\ref{upsilondefinition})
condition (\ref{kpartialassumption}) holds.
Assume that
%
$\|\bar u\|_{C^2([0,L])}$ is sufficiently small
such that
$\|\bar u\|_{C([0,L])} < \varepsilon_1(2\, k^2)$
 and
(\ref{ubarvor2}) holds.
%
Choose $\varepsilon_2$ as in Lemma \ref{normlemma}.
Define $T$ as in (\ref{Tschranke0305}).

Choose a real number  $\tilde \varepsilon_0 \in (0,\,\varepsilon_0(T))$
%
%
sufficiently small such that for all
initial data that satisfies
\begin{align}
\label{4.835}
\|(\varphi(x),\psi(x))\|_{H^2([0,L])\times H^1([0,L])}\leq \tilde \varepsilon_0
\end{align}
and the
$C^1$-compatibility conditions at the points
$(t,x) = (0,0)$ and $(t, x)=(0,L)$
 the
 solution of
 problem (\ref{2.26}),(\ref{2.28}),(\ref{2.29}),(\ref{2.30})
  exists on the time-interval $[0,\, T]$.
Moreover, assume that
$\tilde \varepsilon_0 > 0$ and $\bar u$ are  sufficiently
 small such that
(\ref{uistklein}),
(\ref{i3vor})  and (\ref{kappadefinition}) hold.
%


%


If  the initial data satisfies  (\ref{4.835}),
%
the mixed initial-boundary value problem
(\ref{revisionsystem})
((\ref{2.26}), (\ref{2.28}), (\ref{2.29}), (\ref{2.30}) respectively)
has  a unique
solution
$u \in L^\infty((0,\infty), H^2[0,L])$
with
$u_t \in L^\infty((0,\infty), H^1[0,L])$.
For the number
$\mu \in  [\frac{1}{4 \,{\rm e} L k}, \, \frac{1}{2\,{\rm e} \, L\,  k })$
defined in (\ref{mudefinition})
we have the inequality
\begin{align}\label{4.9h}
E(t)\leq E(0) \, \exp\left(-{\mu}\,t\right) \; \mbox{\rm for all }\; t>0
\end{align}
with the strict Lyapunov function $E(t)$ as defined in (\ref{4.1}).
Moreover, we have
\[E_1(t)\leq E_1(0) \, \exp\left(-{\mu}\,t\right) \; \mbox{\rm for all }\; t>0.\]
\end{thm}

\section{Summary and outlook}
In this paper we have considered a
quasilinear
wave
equation
for the velocity of a gas flow that is governed by
the isothermal Euler
equations with friction.
We have presented a method of boundary
feedback stabilization to stabilize the velocity locally around a
given stationary state.
For the proof, we have  introduced a strict $H^2$-Lyapunov
function (see (\ref{4.1})).


We have shown that, for initial
conditions with sufficiently small $C^2\times C^1$-norm and for
appropriate boundary feedback conditions, the $H^2\times H^1$-norm of the
solution decays exponentially with time.
In addition, we have
shown
that with our velocity  feedback law,
for initial data with
sufficiently small $H^2\times H^1$-norm
the solution exists globally in time and the $H^2\times H^1$-norm of the
solution decays exponentially.

In this paper, the strict $H^2$-Lyapunov function is used to prove
the stability of the solution. It would also be interesting to
consider other types of Lyapunov functions, such as weak Lyapunov
functions. Moreover, when a disturbance is considered,
Input-to-State Stability
Lyapunov functions should be studied
(see \cite{LGM},\cite{MP}).

We have presented our stabilization method for a single pipe
applying an active control at an end of the pipe.
Some additional work is required
to extend this method to more complicated gas networks. For the
stabilization of networks it is often necessary to apply an active
control in the interior of the networks. The well-posedness of
systems of balance laws on networks is studied in \cite{GHKLS}. For
a star-shaped network of vibrating strings governed by the wave
equation, a method of boundary feedback stabilization is presented
in \cite{GS}, where not for each string an active control is
necessary. A related open problem is the feedback stabilization of
more complicated pipe networks with leaks. Moreover, also feedback
stabilization of second-order hyperbolic equations with time-delayed
controls is worth to be studied. For wave equations, this has been
done in \cite{G} and in \cite{NVF} and for the isothermal Euler
equations with an $L^2$-Lyapunov function in \cite{GD}.
In the current paper we have considered an ideal gas with constant
sound speed. It would  be interesting to look at
more realistic models of gas where the sound speed also depends on
the pressure.

\section*{Acknowledgments}
This work is  supported by DFG in the framework of the Collaborative Research Centre
CRC/Transregio 154, Mathematical Modelling, Simulation and Optimization Using the Example of Gas Networks, project C03 and project A03,
by the
DFG-Priority Program 1253: Optimization with partial differential
equation (grant number GU 376/7-1) and the Excellent Doctoral Research
Foundation for Key Subject of Fudan University (No. EHH1411208).
We thank the reviewers for the careful reading and
the comments that helped to improve the paper.


%
%

\end{document}